\documentclass[12pt]{amsart}

\usepackage[a4paper,left=2.7cm,right=2.7cm,top=2.7cm,bottom=2.7cm]{geometry}

\usepackage{amsthm,amsmath,amsfonts,amssymb}
\usepackage{mathrsfs,latexsym,enumitem,stmaryrd}
\usepackage{graphicx,caption,subcaption}
\usepackage[normalem]{ulem}

\newtheorem{theorem}{Theorem}[section]
\newtheorem{lemma}[theorem]{Lemma}
\newtheorem{proposition}[theorem]{Proposition}
\newtheorem{corollary}[theorem]{Corollary}
\newtheorem{remark}[theorem]{Remark}
\numberwithin{equation}{section}

\begin{document}
\title{On the action of multiplicative cascades on measures}

\author{Julien Barral}
\address{Laboratoire d'Analyse, G\'eom\'etrie, et Applications, CNRS, UMR 7539, Universit\'e Sorbonne Paris Nord, CNRS, UMR 7539,  F-93430, Villetaneuse, France}
\email{barral@math.univ-paris13.fr}

\author{Xiong Jin}
\address{Department of Mathematics, University of Manchester, Oxford Road, Manchester M13 9PL, United Kingdom}
\email{xiong.jin@manchester.ac.uk}

\begin{abstract}
We consider the action of Mandelbrot multiplicative cascades on probability measures supported on a symbolic space. For general probability measures, we obtain almost a sharp criterion of non-degeneracy of the limiting measure; it relies on the lower and upper Hausdorff dimensions of the measure and the entropy of the random weights. We also obtain sharp bounds for the lower Hausdorff and upper packing dimensions of the limiting measure. When the original measure is a Gibbs measure associated with a potential of certain modulus of continuity (weaker than H\"older), all our results are sharp. This improves results previously obtained by Kahane and Peyri\`ere, Ben Nasr, and Fan. We exploit our results to derive dimension estimates and absolute continuity for some random fractal measures.
\end{abstract}

\maketitle

\section{Introduction}

Multiplicative chaos  theory \cite{Kahane1987b} was set up by J.-P. Kahane to unify various models of random measures occurring in the study of random Fourier series, random covering problems, as well as modeling of intermittent phenomena (see~\cite{Billard1965,Kahane1985,Mandelbrot1971,Mandelbrot1974}). As particularly remarkable models, one finds the multiplicative cascades on the boundary of a regular tree and their geometric realisations on unit cubes considered by B.~Mandelbrot \cite{Mandelbrot1974}. This model was considered as a simplification of a model of energy dissipation he introduced in his investigations of the log-normal hypothesis in Kolmogorov's work on Turbulence \cite{Kolmogorov1961,Mandelbrot1971}. This model led Kahane to elaborate Gaussian multiplicative chaos theory~\cite{Kahane1985a}, which turns out to play a central role at the interface of probability theory and theoretical physics \cite{DuSh,KRV,BeWeWo}. More generally, log-infinitely divisible multiplicative chaos have been considered \cite{Fan1997a,BM,BaMu,RSV}.

In the general theory~\cite{Kahane1987b}, one defines operators on Radon measures on a locally compact metric space $(T,d)$. Since in this paper $T$ will be compact, let us assume now that it is so.  To define a multiplicative chaos on $T$, consider a sequence of nonnegative measurable functions $(Q_n)_{n\in\mathbb Z_+}$, called $T$-martingale,  defined on the product of a probability space $(\Omega,\mathcal  A,\mathbb P)$ with $(T,\mathcal B(T))$, and such that for each $t\in T$, the sequence $(Q_n(\cdot,t))_{n\in\mathbb{N}}$ is a martingale with respect to the filtration $(\sigma(Q_k(\cdot,t): 0\le k\le n,\, t\in T))_{n\ge 1}$. For each $t\in T$, define $q(t)=\mathbb{E} (Q_n(\cdot,t))$, where $n$ is any nonnegative integer. The sequence $(Q_n)_{n\in\mathbb Z_+}$ defines an operator $\mathcal Q$ on the set $\mathcal M_{+,q}(T)$ of Radon measures $\nu$ on $(T,\mathcal B(T))$ such that $\int_Tq(t)\, \nu(\mathrm{d}t)<\infty$. For simplicity in this paper we will take $q\equiv1$ and denote $\mathcal M_{+,1}(T)$ by $\mathcal M_{+}(T)$. The total mass of an element $\nu\in \mathcal M_{+}(T)$ is denoted by $\|\nu\|$. An operator $\mathcal Q$ on $\mathcal M_{+}(T)$ is then defined in the following way~\cite{Kahane1987b}:

 If $\nu\in \mathcal M_{+}(T)$, due to the assumption made on $(Q_n)_{n\in\mathbb Z_+}$, the sequence of measures $(\nu_n)_{n\in\mathbb Z_+}$ obtained by taking $\nu_n(\mathrm{d} t)=Q_n(t) \cdot\nu(\mathrm{d} t)$ is a $\mathcal M_{+}(T)$-valued martingale of expectation $ \nu$, which almost surely converges weakly to a random measure denoted by $Q\cdot\nu$ as $n\to\infty$. Then, one sets 
$$
\mathcal Q(\nu)=\mathbb{E}(Q\cdot\nu),
$$
where the equality means that $\int_Tf\, {\rm d}\mathcal{Q}(\nu)=\mathbb{E}\int_T f\, {\rm d} (Q\cdot\nu)$ for any real-valued continuous function defined on $T$.  By Fatou's lemma, one has $\|\mathcal Q(\nu)\|\le\liminf_{n\to\infty} \mathbb{E} (\|\nu_n\|)= \|\nu\|$, with equality if and only if the martingale $(\|\nu_n\|)_{n\ge 1}$ is uniformly integrable.  In this case, Kahane says that $\nu$ is $Q$-regular, or that $\mathcal Q$ acts fully on $\nu$. If ever $\mathcal Q(\nu)=0$, i.e., $ Q\cdot\nu=0$ almost surely, he says that $\nu$ is $Q$-singular, or that $\mathcal Q$ kills $\nu$. Of course, the intermediate situation where $Q\cdot \nu\neq 0$ with positive probability but $\nu$ is not $Q$-regular can occur.

In concrete examples, the sequence $(Q_n)_{n\in\mathbb Z_+}$ has a product structure, and $Q$-regularity (resp. singularity) of $\nu$ is also referred to by saying that the multiplicative chaos $Q$ acts fully on (resp. kills) $\nu$.  Finding sharp effective criteria characterising when a multiplicative chaos fully acts on a given measure  is an essential question in this theory. This has been done for some  infinitely divisible multiplicative chaos (starting with the log-normal case) acting on Lebesgue measure restricted to compact domains of $\mathbb R^d$ \cite{Kahane1985,BM,BJ}, as well as for the Mandelbrot multiplicative cascades on the boundary of a regular tree  when they act on the uniform measure \cite{KP1976}, and more generally on so-called Markov measures~\cite{Fan} (see also \cite{BenNasr}), as well as in the context of martingale convergence in branching random walks \cite{Biggins77,Lyons97}.  In the same context, such a criterion is also exhibited in \cite{WW96} for more general cascades acting on the uniform measure. Results, not sharp but quite precise, also exist for general measures, in connection with estimates for their lower Hausdorff dimension \cite{Fan,BM2,FJ}.

In this paper, we focus on multiplicative cascades on the boundary of a $b$-adic tree $\Sigma $ $(b\ge 2)$. Following the size biasing approach first used in this context by Waymire and Williams \cite{WW96}, and then by Lyons~\cite{Lyons97}, and Biggins and Kyprianou \cite{BK04}, which respectively deal with multiplicative cascades acting on the uniform measure $\lambda$ on $\Sigma$, and martingales in the branching random walk, we first consider an abstract model of multiplicative cascade on $\Sigma$, and obtain  sufficient conditions for $Q$-regularity and $Q$-singularity of a given Borel probability measure  (Theorem~\ref{main1}).

Then, we strengthen the existing results on $Q$-regularity/$Q$-singularity for the action of a Mandelbrot multiplicative cascade on general measures on $\Sigma$ (Theorem~\ref{thm1} (a) and (b)), and in case of $Q$-regularity,  we give a sharp estimate of the lower Hausdorff dimension and upper packing dimension of the limit measure when it is non degenerate (Theorem~\ref{thm1} (c)).  When~$\nu$ is exact-dimensional, we can partially reformulate the criterion for $Q$-regularity in terms of the comparison of the dimension of $\nu$ with a kind of entropy associated with the multiplicative cascade, extending what is known when $\nu=\lambda$. Also, we enlarge the class of measures for which we have a full characterisation of $Q$-regularity, by considering the class of ergodic probability measures when $\Sigma$ is interpreted as a symbolic space endowed with the left shift operation. Specifically, for such an ergodic measure, we first notice that it is either $Q$-regular or $Q$-singular, and then for ergodic measures satisfying some Gibbs property, we completely characterise the $Q$-regularity in term of  the entropy of the measure (Theorem~\ref{thm2} and Corollary~\ref{corthm2}; our result  extends to a slightly more general class of multiplicative cascades also considered by Mandelbrot (Corollary~\ref{cor2} and Theorem~\ref{main2})). As a consequence, we treat the case of some invariant measures by inspection of their ergodic decomposition (Corollary~\ref{cor-2.11}). Our results also yield new applications of multiplicative chaos to  the absolute continuity, with respect to their expectation, of the orthogonal projections to subspaces of $\mathbb R^d$ of some random statistically self-similar measures defined on $\mathbb R^d$ (Theorem~\ref{Application2}), as well as to the  computation of the Hausdorff dimension of some random measures on Bedford-McMullen carpets (Theorem~\ref{Application1}). 

\section{Results}
We begin by considering the action of abstract multiplicative cascades. 

\subsection{Abstract sufficient conditions for $Q$-regularity and $Q$-singularity}
\label{sec2.1}

Let $\Lambda=\{0,\ldots, b-1\}$ be an alphabet of $b\ge 2$ letters. Denote by $\Lambda^*=\cup_{n\ge 0} \Lambda^n$ the set of finite words over $\Lambda$, with the convention that $\Lambda^0=\{\emptyset\}$. If $u\in\Lambda^n$ for $n\ge 0$, we denote by $|u|$ the length of $u$, defined to be equal to $n$, and also called generation of $u$.  

Let $\Sigma=\Lambda^\mathbb{\mathbb Z_+}$ be the one-sided symbolic space over the alphabet $\Lambda$. We endow $\Sigma$ with the standard ultra-metric distance $d((x_n)_{n=1}^\infty,  (y_n)_{n=1}^\infty)=e^{-\min\{n\ge 1: \, x_n\neq y_n\}}$, and denote by $\mathcal{B}$ the associated Borel $\sigma$-field. We denote by  $\sigma$  the left shift operation on $\Sigma$. 

For $u\in\Lambda^*$, denote by $[u]$ the set of elements of $\Sigma$ with common prefix $u$. This set is also called the cylinder rooted at $u$. The set of cylinders generates $\mathcal B$. 

Denote by $\lambda$ the uniform measure on $(\Sigma,\mathcal B)$, i.e. the unique probability measure on $\Sigma$, which assigns mass $b^{-|u|}$ to each cylinder $[u]$, $u\in \Lambda^*$. 

For $x=x_1x_2\cdots\in\Sigma$,  denote by $x|_0$ the empty word and by $x|_{n}$ the word $=x_1x_2\cdots x_{n}$ for $n\ge 1$.

Let $(\Omega,\mathcal{A},\mathbb{P})$ be a probability space. Let $\{X_u: u\in\Lambda^*\}$ be a sequence of nonnegative random variables on $(\Omega,\mathcal{A},\mathbb{P})$ indexed by the set of non-empty finite words. 

Denote  the $\sigma$-field generated by $\{X_u:\, |u|\le n\}$ by $\mathcal F_n$:
\[
\mathcal{F}_n=\sigma\{X_u: |u|\le n\}.
\]

For $v\in\Lambda^*$, $n\ge 1$ and $u=u_1\cdots u_n\in \Lambda^n$ denote 
\[
Y_u^{[v]}=X_{vu_1} X_{vu_1u_2}\cdots X_{vu_1\cdots u_n}.
\]
When $v=\emptyset$ we shall write $Y_u^{[\emptyset]}=Y_u$. We set $X_\emptyset\equiv Y_\emptyset^{[v]}\equiv 1$.  

Throughout the paper, we assume that
\medskip
\begin{itemize}
\item[] For all $n\ge 1$ and $u=u_1\cdots u_{n+1}\in \Lambda^{n+1}$, one has $\mathbb{E}(Y_{u_1\cdots u_{n+1}}|\mathcal F_n)=Y_{u_1\cdots u_n}$.
\end{itemize}
\medskip

Consider the product space
\[
(\widetilde{\Omega},\widetilde{\mathcal{A}})=(\Omega\times \Sigma, \mathcal{A}\otimes\mathcal{B}).
\]
For $n\ge 1$ define a random variable $W_{n}$ on $\widetilde{\Omega}$ by
\[
W_{n}(\omega,x)=X_{x|_n}(\omega).
\]
Then, define the random variable $Q_n$ on $\widetilde{\Omega}$ by
\[
Q_{n}=W_{1}W_{2}\cdots W_{n}.
\]
By assumption, $(Q_n)_{n\in\mathbb Z_+}$ is a $\Sigma$-martingale.

Define the $\sigma$-field $\widetilde{\mathcal{F}}_{n}=\sigma\{W_1,\cdots,W_n\}$, and  $\widetilde{\mathcal{F}}=\sigma(\bigcup_{n\ge 1}\widetilde{\mathcal{F}}_{n})$. Let $\nu$ be a Borel probability measure on $(\Sigma,\mathcal{B})$ (since $\Sigma$ is compact this is not a restriction with respect to working with positive Radon measures). Let $\widetilde{\mathbb{P}}$ be the probability measure on $(\widetilde{\Omega},\widetilde{\mathcal{F}})$ such that for  $\widetilde{\mathcal{F}}$-measurable function $f$,
\begin{equation}\label{eq0}
\int_{\widetilde{\Omega}} f(\omega,x) \, \widetilde{\mathbb{P}}({\rm d}\omega,{\rm d}x)=\mathbb{E}_{\mathbb P}\int_{\Sigma} f(\cdot,x) \, \nu({\rm d}x).
\end{equation}
Then let $\widetilde{\mathbb{Q}}$  be the probability measure on $(\widetilde{\Omega},\widetilde{\mathcal{F}})$ such that for $n\ge 1$ and $\widetilde{\mathcal{F}}_{n}$-measurable function $f$,
\begin{equation}\label{eq1}
\int_{\widetilde{\Omega}} f(\omega,x) \, \widetilde{\mathbb{Q}}({\rm d}\omega,{\rm d}x)=\int_{\widetilde{\Omega}} f(\omega,x) Q_{n}(\omega,x) \, \widetilde{\mathbb{P}}({\rm d}\omega,{\rm d}x).
\end{equation}
The existence and uniqueness of $\widetilde{\mathbb{Q}}$ is granted by Kolmogorov consistency theorem.

For $n\ge 1$ denote the total mass of the random measure $\nu_n=Q_n\cdot\nu$ by $Z_n$. We have
\[
Z_n=\sum_{u\in \Lambda^n} Y_u \cdot \nu([u]).
\]

Also, set  $\mathcal{F}=\sigma(\cup_{n\ge 1}\mathcal{F}_n)$. By construction $\{Z_n,\mathcal{F}_n\}_{n\ge 1}$ is a non-negative martingale on $(\Omega,\mathcal{F},\mathbb{P})$ with expectation $1$, therefore it converges $\mathbb{P}$-almost surely. For all $\omega\in\Omega$ and $x\in\Sigma$, define 
\[
\begin{cases}
Z(\omega)=\limsup_{n\to\infty} Z_n(\omega),\\
L(\omega,x)=\sup_{n\ge 0}Y_{x|_{n}}(\omega)\nu([x|_{n}]),\\
S(\omega,x)=\sum_{k=0}^\infty Y_{x|_{k}}(\omega)\nu([x|_{k}]).
\end{cases}
\]
For $x\in \Sigma$ define $\mathcal{H}_x=\sigma\{X_{x|_n}:n\ge 1\}$. We denote by {\bf (M)} the following property: 
\begin{itemize}
\item[{\bf (M)}] For all $x\in\Sigma$, $k\ge 0$ and $u\in\Lambda^*\setminus\{\emptyset\}$ with $u_1\neq x_{k+1}$, $\mathbb{E}(Y_u^{[x|_k]}\, |\, \mathcal{H}_x)=1$, $\mathbb{P}$-almost surely.
\end{itemize}

\begin{theorem}
\label{main1}
The following sufficient conditions for $Q$-singularity and $Q$-regularity  hold:
\begin{itemize}
\item[(i)] If $\widetilde{\mathbb{Q}}(L=\infty)>0$, then $\mathbb{E}_{\mathbb{P}}(Z)<1$. If $\widetilde{\mathbb{Q}}(L=\infty)=1$, then $\mathbb{E}_{\mathbb{P}}(Z)=0$,  i.e.  $\nu$ is $Q$-singular.

\item[(ii)] Suppose that {\bf (M)} holds. If $\widetilde{\mathbb{Q}}(S<\infty)>0$, then $\mathbb{E}_{\mathbb{P}}(Z)>0$, i.e. $Q\cdot\nu$ is non-degenerate. If $\widetilde{\mathbb{Q}}(S<\infty)=1$, then $\mathbb{E}_{\mathbb{P}}(Z)=1$, i.e. $Q$-regular.
\end{itemize}
\end{theorem}

In order to relate this preliminary statement to the existing literature, it is worth providing its proof immediately.

\begin{proof}
For $n\ge 1$ denote by $\widetilde{\mathbb{P}}_n$ and $\widetilde{\mathbb{Q}}_n$ the restrictions of $\widetilde{\mathbb{P}}$ and  $\widetilde{\mathbb{Q}}$ on $\widetilde{\mathcal{F}}_n$ respectively. From \eqref{eq1} we get
\begin{equation}\label{eq2}
\widetilde{\mathbb{Q}}_n({\rm d}\omega,{\rm d}x)=Q_{n}(\omega,x) \, \widetilde{\mathbb{P}}_n({\rm d}\omega,{\rm d}x).
\end{equation}
Let $\pi$ be the projection from $\widetilde{\Omega}$ to $\Omega$. Define $\mathbb{Q}=\widetilde{\mathbb{Q}}\circ\pi^{-1}$. For $n\ge 1$ denote by $\mathbb{P}_n$ and $\mathbb{Q}_n$ the restrictions of $\mathbb{P}$ and $\mathbb{Q}$ on $\mathcal{F}_n$ respectively. By \eqref{eq2} we get
\begin{equation}
\mathbb{Q}_n({\rm d}\omega)=Z_n(\omega) \mathbb{P}_n({\rm d}\omega).
\end{equation}
which implies that $\mathbb{Q}_n\ll \mathbb{P}_n$ with Radon-Nikodym derivative $Z_n$. Hence $\{Z_n,\mathcal{F}_n\}_{n\ge 1}$ is a $\mathbb{P}$-martingale and $\{1/Z_n,\mathcal{F}_n\}_{n\ge 1}$ is a $\mathbb{Q}$-martingale. By \cite[Theorem 4.3.5]{Dur} we have for $A\in\mathcal{F}$,
\begin{equation}\label{Du}
\mathbb{Q}(A)=\int_A Z(\omega) \,\mathbb{P}({\rm d}\omega)+\mathbb{Q}(A\cap \{Z=\infty\}).
\end{equation}
In particular, when $A=\Omega$,
\[
\mathbb{E}_{\mathbb{P}}(Z)=1-\mathbb{Q}(Z=\infty)=\mathbb{Q}(Z<\infty).
\]

Now, we first note that it follows from the definition of $(Z_n)_{n\ge 1}$  that we have for all $n\ge 1$, $\omega\in\Omega$ and $x\in\Sigma$,
\begin{equation}\label{yng}
Z_n(\omega)\ge Y_{x|_{n}}(\omega)\nu([x|_{n}]).
\end{equation}
Due to \eqref{yng}, if $L(\omega,x)=\infty$, then we have $Z(\omega)=\infty$. Hence if $\widetilde{\mathbb{Q}}(L=\infty)>0$ then $\mathbb{Q}(Z=\infty)>0$ and therefore $\mathbb{E}_{\mathbb{P}}(Z)<1$. Similarly if $\widetilde{\mathbb{Q}}(L=\infty)=1$ then $\mathbb{Q}(Z=\infty)=1$ and $\mathbb{E}_{\mathbb{P}}(Z)=0$.

Next, we note that for each $x\in \Sigma$ we may rewrite $Z_n$ as
\begin{equation}\label{yn}
Z_n(\omega)=Y_{x|_{n}}(\omega)\nu([x|_{n}])+\sum_{k=0}^{n-1} Y_{x|_{k}}(\omega)\sum_{u\in \Lambda^{n-k}: u_1\neq x_{k+1}} Y_{u}^{[x|_k]}(\omega)\nu([x|_{k}\cdot u])
\end{equation}
A similar decomposition is used in \cite{WW96} when $\nu=\lambda$. However it is exploited in a different way from what we are going to do, which borrows some idea from~\cite{BK04}. 

Since $\{1/Z_n,\mathcal{F}_n\}_{n\ge 1}$ is a positive $\mathbb{Q}$-martingale, $Z_n$ converges to $Z$, $\mathbb{Q}$-almost surely.  Then, using Fatou's lemma, as well as \eqref{yn}  and {\bf (M)} we can get 
\begin{align*}
\mathbb{E}_{\mathbb Q}(Z \, |\,  \mathcal{H}_x)(\omega)
\le&\ \liminf_{n\to\infty}\mathbb{E}_{\mathbb{Q}}(Z_n \, |\,  \mathcal{H}_x)(\omega)\\
=&\ \liminf_{n\to\infty} \Big \{Y_{x|_{n}}(\omega)\nu([x|_{n}])+ \sum_{k=0}^{n-1} Y_{x|_{k}}(\omega) (\nu([x|_k])-\nu([x|_{k+1}]))\Big \}\\
\le & \ S(\omega,x)= \sum_{k=0}^\infty Y_{x|_{k}}(\omega)\nu([x|_{k}]).
\end{align*}
Hence if $\widetilde{\mathbb{Q}}(S<\infty)>0$ then $\mathbb{Q}(Z<\infty)>0$ and therefore $\mathbb{E}_P(Z)>0$. Similarly if $\widetilde{\mathbb{Q}}(S<\infty )=1$ then $\mathbb{Q}(Z<\infty)=1$ and therefore $\mathbb{E}_{\mathbb{P}}(Z)=1$. 
\end{proof}

\begin{remark}
\label{rem0}
(i) The probability measure $\widetilde{\mathbb{Q}}$ was first introduced in \cite{KP1976} by Peyri\`ere in the case of Mandelbrot multiplicative cascades (i.e. the $X_u$, $u\in\Lambda^*\setminus\{\emptyset\}$, are i.i.d) acting on~$\lambda$, and when the action of the chaos is known to be full; then $\widetilde {\mathbb Q}({\rm d}\omega,dx)$ can be defined as $\mathbb{P}({\rm d}\omega) (Q\cdot \lambda(\omega))({\rm d}x)$, and Peyri\`ere used it to compute the Hausdorff dimension  of $Q\cdot \lambda$. The version used in this paper follows that of  Waymire and Williams in \cite{WW96}, who use Kolmogorov extension theorem to get $\widetilde{\mathbb{Q}}$ without requiring  $\mathbb{E}_\mathbb{P}(Z)=1$, but still work with $\lambda$.

\medskip

\noindent (ii) The proof of Theorem~\ref{main1} uses a mixture of  the approach provided in \cite{WW96,Lyons97,BK04}. 

In \cite{WW96}, the authors consider the action on $\lambda$ of multiplicative cascades which obey condition~{\bf (M)}, and possess the additional property that the $X_u$ indexed by words of the same generation are independent. For these multiplicative cascades, Theorem~\ref{main1} provides an extension, valid for any $\nu$, of \cite[Corollary 2.3]{WW96} (note that, using  \cite{WW96} terminology, here we consider cascades with no additional weight system). Also, the proof  is more direct.

In \cite{Lyons97} and \cite{BK04},  the so-called additive martingales, as well as more general martingales associated with  branching random walks are considered; the idea of conditioning with respect to $\mathcal H_x$ comes from~\cite{BK04}. In our context, the additive martingales correspond to $(Z_n)_{n\ge 1}$ in the case where $\nu$ is a Bernoulli product measure, and the $X_u$, $u\in\Lambda^*\setminus\{\emptyset\}$, have the additional property that the random vectors $(X_{ui})_{0\le i\le b-1}$, $u\in\Lambda^*$, are independent and with positive entries; in this case, Theorem~\ref{main1} is essentially a restatement of \cite[Theorem 12.2]{BK04}.

\medskip

\noindent (iii) If $L(\omega,x)=\infty$ then $S(\omega,x)=\infty$, so it is natural to wonder under which conditions on $Q$ and $\nu$ one can replace $L(\omega,x)=\infty$ by $S(\omega,x)=\infty$ in the statement of Theorem~\ref{main1} and get the same conclusion.  In the next sections we will exhibit situations in which this is possible, and for which the alternative $S(\omega,x)=\infty$ $\widetilde {\mathbb{Q}}$-a.e./$S(\omega,x)<\infty$  $\widetilde {\mathbb{Q}}$-a.e. can be expressed in a simpler way and does hold. In particular, this will show that the property $\widetilde {\mathbb{Q}}(\{L=\infty\})>0$ is not necessary for $\nu$ not to be $Q$-regular (see the comment following Theorem~\ref{cor1} below). On the other hand, in Section~\ref{sec2.4} we give an example where $\nu$ is $Q$-regular but $\widetilde {\mathbb{Q}}(\{S=\infty\})>0$ (see Remark~\ref{special case}). 
\end{remark}

In the next section we consider the specific class of Mandelbrot multiplicative cascades and simplify the criterion of non-degeneracy provided by Theorem~\ref{main1}. We also give a sharp result regarding the lower Hausdorff and upper packing dimensions of the limit measure. 

\subsection{Results for Mandelbrot multiplicative cascades}\label{secMMC}

We need a few additional definitions. 

The lower and upper local dimensions of a positive Radon measure $\nu$ on $\Sigma$ at a point $x$ of the topological support of $\nu$ are respectively defined by 
$$
\underline\dim_{{\rm loc}}(\nu,x)=\liminf_{n\to\infty}\frac{-\log \nu([x|_n])}{n} \text{ and }\overline\dim_{{\rm loc}}(\nu,x)=\limsup_{n\to\infty}\frac{-\log \nu([x|_n])}{n},
$$
and in case of equality, the common value is called local dimension of $\nu$ and denoted by $\dim_{{\rm loc}}(\nu,x)$. 

The lower and upper Hausdorff dimensions, as well as the upper packing dimension of $\nu$ are respectively defined as follows (see, e.g., \cite{Fan1989c,Fan1994,Heurteaux1998})
$$
\underline{\dim}_H (\nu)=\inf\{\dim_H E:\,  E\text{ Borel set},\, \nu(E)>0\}={\mathrm{ess}\inf}_\nu\, \underline\dim_{{\rm loc}}(\nu,\cdot),
$$
$$
\overline{\dim}_H (\nu)=\inf\{\dim_H E:\,  E\text{ Borel set},\, \nu(\Sigma\setminus E)>0\}={\mathrm{ess}\sup}_\nu\,\underline\dim_{{\rm loc}}(\nu,\cdot),
$$
and 
$$
\overline{\dim}_P (\nu)=\inf\{\dim_P E:\,  E\text{ Borel set},\, \nu(\Sigma\setminus E)>0\}= {\mathrm{ess}\sup}_\nu\,\overline\dim_{{\rm loc}}(\nu,\cdot),
$$
where $\dim_H$ and $\dim_P$ stand for the Hausdorff and packing dimensions respectively. If there exists $D>0$ such that $\underline{\dim}_H (\nu)=\overline{\dim}_P (\nu)=D$, or equivalently $\dim_{{\rm loc}}(\nu,x)=D$ $\nu$-almost everywhere, one says that $\nu$ is exact-dimensional with dimension $D$, and we write $\dim(\nu)=D$. 

In the case when property {\bf (M)} is strengthened into
\begin{itemize}
\item[{\bf (M1)}] $\{X_u: u\in\Lambda^*\setminus \{\emptyset\}\}$ are i.i.d. random variables with a common law $X$, where $X\ge 0$ is a positive random variable with $\mathbb{E}(X)=1$,
\end{itemize}
the sequence $(Q_n)_{n\in\mathbb Z_+}$ is called a Mandelbrot multiplicative cascade. 

This model of multiplicative chaos has been first considered in a model for Turbulence in \cite{Mandelbrot1974} when it acts on  the uniform measure $\lambda$ on $\Sigma$. Note that $\lambda$ is exact-dimensional with $\dim (\lambda)=\log (b)$. A necessary and sufficient condition for this measure to be $Q$-regular was obtained in \cite{KP1976}. Specifically, setting
$$
h_X=\mathbb E(X\log X)\in[0,\infty],
$$
(the above expectation is well-defined since $x\log x \ge -e^{-1}$ on $\mathbb{R}_+$) then $\lambda$ is ${Q}$-regular or $Q$-singular according to whether $h_X< \log b=\dim (\lambda)$ or not. Also, it is proved in \cite{KP1976, Kahane1987} that, conditional on $Q\cdot\lambda\neq 0$, $Q\cdot\lambda$ is exact-dimensional with dimension $\log (b)-h_X=\dim (\lambda)-h_X$.

Such optimal results were extended to Markov measures in \cite[Theorem A]{Fan},  i.e. $\sigma$-ergodic Gibbs measures associated with a  potential constant over the cylinders of generation $\ell$  for some $\ell \in\mathbb Z_+$. Results, less sharp, valid for general measures are obtained in \cite[Theorem B]{Fan}.

\subsubsection{\bf Action of multiplicative cascades on general measures under {\bf (M1)}}
\label{sec2.2}
\begin{theorem}\label{thm1} 
Let $\nu$ be a Borel probability measure on $\Sigma$. 
\begin{itemize}
\item[(a)] If for $\nu$-almost every $x\in \Sigma$,
\begin{equation}\label{sup}
\underline\dim_{{\rm loc}}(\nu,x)>h_X,
\end{equation}
then  $\nu$ is $Q$-regular. This holds in particular if $\underline \dim_H (\nu)>h_X$. 
\item[(b)] If for $\nu$-almost every $x\in \Sigma$,
\begin{equation}\label{inf}
\underline\dim_{{\rm loc}}(\nu,x)<h_X,
\end{equation}
then $\nu$ is $Q$-singular. This holds in particular if $\overline \dim_H (\mu)<h_X$. 
\item[(c)] Under the assumption of (a), with probability 1, conditional on $Q\cdot\nu\neq0$, 
$$
 \underline \dim_H (\nu)-h_X\le \underline \dim_H (Q\cdot\nu)\le \overline \dim_P (Q\cdot\nu)\le \overline \dim_P (\nu)-h_X.
$$ 
\end{itemize}
\end{theorem}

\begin{corollary}\label{co-ed}
Let $\nu$ be a Borel probability measure on $\Sigma$.  Assume that $\nu$ is exact-dimensional with dimension $D>h_X$. With probability 1, conditional on $Q\cdot\nu\neq0$, $Q\cdot\nu$ is exact-dimensional with dimension $D-h_X$. 
\end{corollary}

\begin{remark}
Theorem~\ref{thm1}  improves significantly Theorem B in \cite{Fan}. Indeed, define the $L^q$-spectrum of $\nu$ as  
\begin{equation}\label{Lqs}
\tau_\nu(q)=\liminf_{n\to\infty} -\frac{1}{n}\log\sum_{u\in\Lambda^n} \nu([u])^q, \quad q\ge 0.
\end{equation}
It is a concave non decreasing function. The following inequalities, which can be strict,  hold: $\tau_\nu'(1+)\le \underline \dim_H(\nu)\le \overline \dim_P(\nu)\le \tau_\nu'(1-)$ (see \cite{Ngai}). In \cite[Theorem B(a)]{Fan} it is assumed that $\mathbb{E}_{\mathbb{P}}(X^q)<\infty$ for some $q>1$, and in place of  condition \eqref{sup} one finds  $\tau_\nu'(1+)>h_X$, while in \cite[Theorem B(b)]{Fan}, in place of condition \eqref{inf} one finds $\tau_\nu'(1-)<h_X$. Also, the proof of \cite[Theorem B(c)]{Fan}  yields the weaker estimates $
  \tau_\nu'(1+)-h_X\le \underline \dim_H (Q\cdot\nu)\le \overline \dim_P (Q\cdot\nu)\le  \tau_\nu'(1-)-h_X
$ (in fact the statement of \cite[Theorem B(c)]{Fan} gives these bounds only for $\underline \dim_H (Q\cdot\nu)$). 
\end{remark}

For any probability measure $\nu$ on $\Sigma$ we may consider the sets
\[
\Sigma_{h_X,\nu}^+=\{x\in\Sigma:
\underline\dim_{{\rm loc}}(\nu,x)>h_X\}.
\]
and
\[
\Sigma_{h_X,\nu}^-=\{x\in\Sigma:
\underline\dim_{{\rm loc}}(\nu,x)<h_X\}.
\]
Items (a) and (b) of Theorem~\ref{main1} can then be precised as follows:
\begin{theorem}\label{cor1}Let $\nu$ be a Borel probability measure on $\Sigma$. 
\begin{itemize}
\item[(1)] If $\nu(\Sigma_{h_{X},\nu}^+)>0$ then $\mathbb{E}_{\mathbb P}(Z)>0$ and if $\nu(\Sigma_{h_X,\nu}^+)=1$ then $\mathbb{E}_{\mathbb P}(Z)=1$.
\item[(2)] If $\nu(\Sigma_{h_{X},\nu}^-)>0$ then $\mathbb{E}_{\mathbb P}(Z)<1$ and if $\nu(\Sigma_{h_X,\nu}^-)=1$ then $\mathbb{E}_{\mathbb P}(Z)=0$.
\end{itemize}
\end{theorem}

The \textit{critical case} when
\[
\nu(\{x\in\Sigma:\, \underline\dim_{{\rm loc}}(\nu,x)=h_X\})>0
\]
is much more delicate.  Indeed if $\alpha=h_X\in (0,\log(b))$, there may exist  exact-dimensional measures $\nu_1$ and $\nu_2$ of Hausdorff dimension $\alpha$ such that   $\nu_1$ is $Q$-regular while  $\nu_2$ is $Q$-singular. This is the case if $P(X=e^\alpha)=e^{-\alpha}$ and $P(X=0)=1-e^{-\alpha}$; denote then $Q$ by $Q_\alpha$. In this case, Kahane showed that for $\nu$ to be $Q_\alpha$-regular,  it is enough that $\nu$ be of  finite energy with respect to the Riesz kernel $k_\alpha(x,y)=d(x,y)^{-\alpha}$ and that  for $\nu$ to be $Q_\alpha$-singular, it is enough that $\nu$ be supported by a set of finite $\alpha$-dimensional Hausdorff measure~\cite{Kahane1987}. Later, Fan proved that $\nu$ is $Q_\alpha$-regular if and only if  $\nu$ is the countable sum of Borel measures of  finite energy with respect to $k_\alpha$, and $\nu$ is $Q_\alpha$-singular if and only if $\nu$ is supported by a Borel set of null capacity with respect to $k_\alpha$~\cite{Fan1989c,Fan1990} (see also \cite{Fan2004}, and \cite{Lyons1990} for the case $\nu=\lambda$). It is worth noting that this exactly corresponds to the fact that $\nu$ is $Q_\alpha$-regular if and only if $\int_\Sigma k_\alpha(x,y)\,\nu({\rm d}y)<\infty$ for $\nu$-a.e. $x$, and $\nu$ is $Q_\alpha$-singular if and only if $\int_\Sigma k_\alpha(x,y)\,\nu({\rm d}y)=\infty$ for $\nu$-a.e. $x$, which precisely corresponds to the alternative $\widetilde {\mathbb Q}_\alpha (\{S<\infty\})=1$/$\widetilde {\mathbb Q}_\alpha (\{S=\infty\})=1$ in Theorem~\ref{main1}. 

Note also that in the case of $\alpha=\log(b)$ and $\nu$ is the uniform measure, one has  $\widetilde {\mathbb Q}_\alpha (\{S=\infty\})=1$ but $L$ is bounded, and $\nu$ is $Q_{\log(b)}$-singular as $Q_{\log(b)}\cdot \nu$ is supported by a degenerate critical Galton-Watson tree. Moreover, for any $\alpha\in (0,\log(b)]$, one has $\widetilde {\mathbb Q}_\alpha (\{L=\infty\})=1$ if and only if 
\begin{equation}\label{div}
\nu\big (\limsup_{n\to\infty} \nu([x|_n])e^{n\alpha}=\infty\big )=1.
\end{equation}

\medskip

We are going to give a sufficient condition for  $Q$-singularity in the critical case when $\nu$ is $\sigma$-ergodic. When $Q=Q_\alpha$, this condition will be equivalent to \eqref{div}, which is easier to check than $\int_\Sigma k_\alpha(x,y)\,\nu({\rm d}y)=\infty$ for $\nu$-a.e.~$x$. 

\subsubsection{\bf Action of multiplicative cascades on ergodic measures under {\bf (M1)}}
\label{sec2.3}

Let us denote by $\mathcal  M(\sigma)$ and $\mathcal E(\sigma)$  the sets of $\sigma$-invariant and $\sigma$-ergodic Borel probability measures on $\Sigma$.  They are endowed with the weak-$*$ topology. We start with a preliminary observation, whose proof we give immediately.

\begin{proposition}If $\nu\in\mathcal E(\sigma)$, then it is either $Q$-regular or $Q$-singular. 
\end{proposition}

\begin{proof} The measure $\mathbb{E}(Q\cdot\nu)$ is both $\sigma$-invariant and absolutely continuous with respect to $\nu$. Since $\nu$ is $\sigma$-ergodic, there exists $c\in [0,1]$ such that $\mathbb{E}(Q\cdot\nu)=c\nu$. Moreover, the operator $ \mu\mapsto \mathbb{E}(Q\cdot\mu)$ is a projection (see~\cite[Theorem 4]{Kahane1987b}). Consequently $c\in\{0,1\}$. 
\end{proof}

Now let  $\nu\in \mathcal E(\sigma)$.  By Shannon-McMillan-Breiman theorem,  for $\nu$-almost every $x\in \Sigma$,
\[
\dim_{{\rm loc}}(\nu,x)=h_\nu(\sigma),
\]
where
\[
h_\nu(\sigma)=\lim_{n\to \infty}\frac{1}{n}  \sum_{u\in \Lambda^n} -\nu([u])\log\nu([u])
\]
is the measure-theoretic entropy of $\nu$. Due to our choice for the metric on $\Sigma$, we have $\dim (\nu)=h_\nu(\sigma)$, hence Theorem~\ref{thm1} already shows that if $h_\nu(\sigma)>h_X$ then $\nu$ is $Q$-regular, while if $h_\nu(\sigma)<h_X$ it is $Q$-singular.

It is natural to raise the following question:

\medskip

{\bf Q1:} if $\mathbb P(X\neq 1)>0$ and   $\nu\in\mathcal E(\sigma)$, does $h_\nu(\sigma)=h_X$ imply that $\nu$ is $Q$-singular? 

\medskip

Indeed, as we already recalled it, this is the case when $\nu$ is the measure of maximal entropy \cite[Th\'eor\`eme 1]{KP1976}, and more generally when $\nu$ is a Markov measure \cite[Theorem A]{Fan}, as was shown by Kahane and Fan respectively. However, when a general ergodic measure $\nu$ is involved this is not clear. According to Theorem~\ref{main1}, one has to  study the asymptotic behaviour of $Y_{x|_n}(\omega)\nu([x|_n])$ (this is not the approach followed by Kahane and Fan, who used functional equations which cease to exist in the general case). This amounts to understand the fluctuations of the summation of a centered random walks only associated with the multiplicative cascade, with  a term which must be thought of as a centered Birkhoff sum depending  only on $\nu$:
\[
\Big (\sum_{k=1}^n \big (\log (W_k(\omega,x))-h_\nu(\sigma)\big )\Big )+ \left (nh_\nu(\sigma)+\log \nu([x|_n])\right ).
\]
When $h_X\neq h_\nu(\sigma)$, the random walk is not centered anymore and it is not hard to derive the conclusion of Theorem~\ref{thm1}(a) and $(b)$ from Theorem~\ref{main1}. Otherwise, to go beyond Kahane's and Fan's results, we need to make some additional assumptions. 

For instance, assume that $\mathbb{E}_\mathbb{P}(X(\log X)^2)<\infty$ and $X$ is not almost surely equal to a constant conditional on $X\neq 0$. In this case $\Big (S_n(\omega,x)=\sum_{k=1}^n \log W_k(\omega,x)-h_\nu(\sigma)\Big )_{n\ge 0}$ satisfies a central limit theorem (CLT) with positive variance. Then, we may meet the following  situation: $(S'_n(\omega,x)=nh_\nu(\sigma)+\log \nu([x|_n]))_{n\ge 0}$ satisfies a CLT as well. This would imply that $L(\omega,x)=\infty$ $\widetilde {\mathbb{Q}}$-almost surely. This situation holds  when $\nu\neq \lambda$ is a Gibbs measure associated with a continuous potential which is of small enough modulus of continuity (see Remark~\ref{rem2.10}(5)). 

However, we want both to avoid additional moment assumptions on the law of $X$, and to consider a wider class of ergodic measures, though this class already covers that of Markov measures considered by Fan. 

\medskip In order to control the fluctuation on the $nh_\nu(\sigma)+\log \nu([x|_n])$ part, we will  assume a Gibbs like property for $\nu$, in the following weak sense:
\begin{itemize}
\item[{\bf(G)}] There exist measurable functions $C\ge 1$ and  $f$ defined on $(\Sigma,\mathcal{B})$, finite  $\nu$-almost everywhere, and such that for $\nu$-almost every $x\in\Sigma$, for all $n\ge 1$:
\[
C(x)^{-1} e^{\sum_{k=0}^{n-1}f(\sigma^kx)}\le \nu([x|_n]) \le C(x) e^{\sum_{k=0}^{n-1}f(\sigma^kx)}.
\]
\end{itemize}
For $x\in\Sigma$ and $n\ge 1$ we may define $g_n(x)=\frac{\nu([x|_n])}{\nu([\sigma x|_{n-1}])}$, then by measure differentiation theorem or martingale convergence theorem, we have that for  $\nu$-almost every $x$, $g_n(x)$ converges to a measurable function $g$. Then we may take
\[
f(x)=\log g(x) \text{ and } C(x)=\exp \left(\sup_{n\ge 1}\left \{\left |\sum_{k=0}^{n-1} \log g(\sigma^kx)-\log(g_{n-k}(\sigma^kx))\right|\right\}\right),
\]
provided $C(x)$ is finite $\nu$-almost everywhere. However in general it is very hard to analyse the convergence speed of $g_n$ to $g$, or in other words, the convergence speed in Shannon-McMillan-Breiman theorem, in order to deduce the finiteness of $C$.

Let us mention that at the moment, in the examples of ergodic measures at our disposal and which satisfy {\bf(G)}, one has $\|C\|_\infty<\infty$. This implies that the measure $\nu$ is {\it quasi-Bernoulli}, namely there exists a constant $K\ge 1$ such that  for all finite words $u$ and $v$, one has 
$$
K^{-1}\, \nu([u])\nu([v])\le \nu([u\cdot v])\le K\, \nu([u])\nu([v]).
$$
Note also that if $f$ is continuous, condition {\bf (G)} is stronger than the weak Gibbs property fulfilled by the eigenvectors of the dual of the Ruelle-Perron-Frobenius operator associated with $f$ (see \cite{Yuri00}).

\medskip

When $h_X=h_\nu(\sigma)$ and $Q$ takes the form $Q_\alpha$, then $\alpha=h_\nu(\sigma)$ and we will assume the following stronger property on $\nu$:
\begin{itemize}
\item[\textbf{(}$\mathbf{\widetilde G}$\textbf{)}:] $\nu$ satisfies {\bf (G)}, and if $\nu\neq \lambda$, then  
\begin{equation}\label{Gtilda}
\nu\big (\limsup_{n\to\infty} \sum_{k=0}^{n-1} \big (f(\sigma^kx)+h_\nu(\sigma)\big )=\infty\big )=1.
\end{equation}
\end{itemize}

Let us state our result and then make some comments on {\bf (G)} and \textbf{(}$\mathbf{\widetilde G}$\textbf{)}.

\begin{theorem}\label{thm2} Suppose that $\mathbb P(X\neq 1)>0$, so that $\mathcal Q$ is not the identity operator.  Let $\nu$ be an ergodic measure satisfying {\bf (G)}; if $Q=Q_{h_\nu(\sigma)}$, suppose that $\nu$ also satisfies  \textbf{(}$\mathbf{\widetilde G}$\textbf{)}.  Then,  if $h_\nu(\sigma)=h_X$,  the measure  $\nu$ is $Q$-singular.\end{theorem}

\begin{corollary}\label{corthm2} Suppose that $\mathbb P(X\neq 1)>0$. Let $\nu$ be an ergodic measure satisfying {\bf (G)}; if  $Q=Q_{h_\nu(\sigma)}$ and $h_X=h_\nu(\sigma)$, suppose  that $\nu$ also satisfies \textbf{(}$\mathbf{\widetilde G}$\textbf{)}. Then, $\nu$ is $Q$-regular if and only if $h_X<h_\nu(\sigma)$.
\end{corollary}

\begin{remark}\label{rem2.10}(1) One necessarily has $h_X>0$ since  $\mathbb E(X)=1$ and $\mathbb P(X\neq 1)>0$. Thus necessarily  $h_\nu(\sigma)>0$ in Theorem~\ref{thm2}.

\medskip

\noindent 
(2) Our proof,  when $Q$ is not of the form $Q_\alpha$, will use, among other arguments, a combination of the size biasing approach recalled in Section~\ref{sec2.1}, with the so-called filling scheme, and we will also need to use the natural extension of~$\nu$ to the two sided symbolic space $\Lambda^{\mathbb Z}$. 

\medskip

\noindent 
(3) Property {\bf (G)} holds for any Gibbs measure associated with a continuous potential satisfying  Walters' condition \cite{Walters}, which guaranties the so-called bounded distortion property to hold. Note that Walters' condition strictly covers the class of potentials $f$ we previously considered, for which the Birkhoff sums of $f$ satisfy a CLT. In particular, Theorem~\ref{thm2} combined with Theorem~\ref{thm1} improves \cite[Theorem A]{Fan} that deals with the Markov measures, which are Gibbs for a potential $f(x)$ depending on the $\ell$ first letters of $x$ for some integer $\ell\ge 0$ (the previous observation using the CLT provides a second alternative proof when $Q\neq Q_{h_\nu(\sigma)}$). 

Note also that due to the bounded distortion property,  the function $C(x)$ can be taken equal to a constant. 
\medskip

\noindent (4) Property {\bf (G)} also holds for the Gibbs states associated to the norm of certain cocycles of $SL_2(\mathbb R)$, for which the potential $f$ is $\nu$-almost everywhere continuous,  and proven to be  in $L^1(\nu)$; also $C(x)$ can be taken equal to a constant in these examples (see \cite[Theorem 2.8]{BKM}). 

We have no example illustrating condition {\bf (G)} in full generality, i.e. with $\|C\|_\infty=\infty$. In particular, an inspection of the examples of non unique equilibrium states considered by Hofbauer in~\cite{Hofbauer1977} (see also~\cite{Ledrappier1977}) shows that they do not satisfy {\bf (G)}.

\medskip
\noindent
(5) Property \textbf{(}$\mathbf{\widetilde G}$\textbf{)} holds when $\nu$ is the equilibrium state of a continuous potential whose modulus of continuity $\omega$ statisfies $\omega (\delta)=O(|\log(\delta)|^{-\gamma})$ with $\gamma>2$, because in this case $\nu$ satisfies \textbf{(}$\mathbf{G}$\textbf{)}, and either $\nu=\lambda$, or the potential is not cohomologous to a constant in $L^2(\Sigma,\nu)$ and $n^{-1/2}\big (\sum_{k=0}^{n-1}(f(\sigma^kx)-h_\nu(\sigma))\big )$ converges in law (under $\nu$) to a non-degenerate centered Gaussian random variable (in other words CLT holds with a positive variance; see \cite{Liv1996,FanJiang2001}). 

Now recall the discussion started just after Theorem~\ref{cor1}.  Since when $\nu$ satisfies {\bf (G)} and differs from~$\lambda$ property~\eqref{Gtilda} is equivalent to property~\eqref{div}, when $\nu$ satisfies \textbf{(}$\mathbf{\widetilde G}$\textbf{)} Theorem~\ref{thm2} restates the fact that setting $\alpha= h_\nu(\sigma)\, (\in (0,\log(b)])$,  $\nu$ is $Q_\alpha$-singular measure. In particular, for $\nu$-almost every $x$ one has $\int_\Sigma \frac{\nu({\rm d}y)}{d(x,y)^\alpha}=\infty$, so the  $k_\alpha$-energy of $\nu$ is infinite:
\[
\iint_{\Sigma\times\Sigma} \frac{\nu(\mathrm{d}x)\nu(\mathrm{d}y)}{d(x,y)^{\alpha}}=\infty.
\]
It is not clear whether one can deduce the above equation directly from \eqref{Gtilda}. 
When $\nu$ is quasi-Bernoulli, this result can be checked directly since the almost multiplicativity property of $\nu$ implies that the $\alpha$-energy of $\nu$ is bounded from below by a constant times the series $\sum_{n\ge 1} e^{n(\alpha-\tau_\nu(2))}$, and  $\alpha=h_\nu(\sigma)\ge \tau_\nu(2)$. Indeed, the $L^q$-spectrum of $\nu$ is concave and differentiable with $\tau_\nu(1)=0$ and $\tau_\nu'(1)=h_\nu(\sigma)$ (see \cite{Heurteaux1998}). 

Moreover, any $\sigma$-ergodic measure $\nu$ is supported on its set of generic points $\mathcal G(\nu)$, and if $\nu\neq\lambda$, any Borel subset of $\mathcal G(\nu)$ of positive $\nu$-measure is of infinite $h_\nu(\sigma)$-dimensional Hausdorff measure \cite{MaWen}, so that $\nu$ cannot be supported on a set of finite $h_\nu(\sigma)$-dimensional Hausdorff measure. 

\medskip
\noindent
(6) It would be interesting to answer the following question: 

\medskip
{\bf Q2:} does any $\sigma$-ergodic measure $\nu$ with positive entropy have null $k_{h_\nu(\sigma)}$-capacity? 

Note that a positive answer to {\bf Q1} implies the same for {\bf Q2}.
\end{remark}

\begin{remark}
When $\nu=\lambda$, Waymire and Williams have exhibited multiplicative cascades of Markov type for which there exists an entropy  parameter $H$ such that for every $x\in\Sigma$, one has $\lim_{n\to\infty} -n^{-1}\log Q_n(\omega,x)=H$ for $\mathbb{Q}$-almost every $\omega$, $\lambda$ is $Q$-regular if $H<\log(b)=h_\lambda(\sigma)$, $Q$-singular if $H>\log(2)$, and depending on the parameters defining $(Q_n)_{n\ge 1}$, in the critical case $H=\log(b)$, $\lambda$ may be $Q$-singular or $Q$-regular (see \cite[Section~5]{WW96}). 
\end{remark}

We end this section with a consequence of Theorem~\ref{thm1} for the action of $Q$ on invariant measures.  

\medskip

For $\alpha\in [0, \log (b)]$ let 
$$
\overline{\mathcal E}(\alpha )=\{\nu\in \mathcal  E(\sigma): h_\nu(\sigma)>\alpha\}\quad\text{and }\quad \underline {\mathcal E}(\alpha )=\{\nu\in \mathcal  E(\sigma): h_\nu(\sigma)<\alpha\}.
$$ These sets are Borel sets in $\mathcal M(\sigma)$ endowed with the topology of weak convergence, since the measure-theoretic entropy is semi-continuous ($\sigma$ is expanding). 

\begin{corollary}\label{cor-2.11} Suppose that $\mathbb P(X\neq 1)>0$. Let $\rho$ be a probability measure on $\mathcal  E(\sigma)$. The measure $\displaystyle \int_{\overline {\mathcal E} (h_X)} \nu\, \rho (\mathrm{d}\nu)$ is $Q$-regular, while the measure $\displaystyle\int_{\underline {\mathcal E} (h_X)} \nu\, \rho (\mathrm{d}\nu)$ is $Q$-singular.

Moreover, 
\begin{equation}\label{regular-inv}
Q\cdot \displaystyle \int_{\overline {\mathcal E} (h_X)} \nu\, \rho (\mathrm{d}\nu)=\displaystyle \int_{\overline {\mathcal E} (h_X)} Q\cdot \nu\, \rho (\mathrm{d}\nu),
\end{equation}
$\mathbb{P}$-almost surely.
\end{corollary}
\begin{remark}
(1) The previous result could be sharpened if one could answer~{\bf Q1} positively. Indeed, in this case, setting $\underline{\widetilde {\mathcal E}}(\alpha)=\{\nu\in \mathcal  E(\sigma): h_\nu(\sigma)\le \alpha\}$, the measure  $\displaystyle \int_{\mathcal E(\sigma)} \nu\, \rho (\mathrm{d}\nu)$ would decompose as a sum of the $Q$-regular measure $\displaystyle \int_{\overline {\mathcal E} (h_X)} \nu\, \rho (\mathrm{d}\nu)$ and the  $Q$-singular measure  $\displaystyle\int_{\underline {\widetilde{\mathcal E}}(h_X )} \nu\, \rho (\mathrm{d}\nu)$. By \cite[Theorem 4]{Kahane1987b}, we know that such a decomposition would be unique. 

\medskip

\noindent
(2) Of course, if $\rho(\{\nu\in \mathcal E(\sigma): h_\nu(\sigma)=h_X\})=0$, the previous decomposition holds.  
\end{remark}

The next section provides an extension of the previous results on the action of Mandelbrot multiplicative cascades on ergodic measures. 

\subsubsection{\bf The action of independent multiplicative cascades on ergodic measures} \label{independent cascade}\label{sec2.4}

Now we make a slightly weaker assumption:
\begin{itemize}
\item[{\bf (M2)}] $\{(X_{uj})_{j\in\Lambda}: u\in\Lambda^*\}$ are i.i.d. random vectors with a common law $V$,
\end{itemize}
where $V=(V_0,\ldots,V_{b-1})$ is a random vector with $V_j\ge 0$ and $\mathbb{E}(V_j)=1$ for $j=0,\ldots,b-1$. Note that $V_j$, $j=0,\ldots,b-1$, are not necessarily independent.

Define 
\[
h_{V,\nu}=\sum_{j\in\Lambda} \mathbb{E}(V_j\log V_j)\nu([j]).
\]

Theorem~\ref{main1} has the following corollary.
\begin{corollary}\label{cor2}\ 
\begin{itemize}
\item[(i)] If $h_{V,\nu}<h_\nu(\sigma)$ then $\mathbb{E}_{\mathbb P}(Z)=1$, i.e., $\nu$ is $Q$-regular.
\item[(ii)] If  $h_{V,\nu}>h_\nu(\sigma)$ then $\mathbb{E}_{\mathbb P}(Z)=0$, i.e., $\nu$ is $Q$-singular.
\end{itemize}
\end{corollary}

Under {\bf (M2)}, we say that $Q$ is of $Q_\alpha$-type if there exists $\alpha>0$ such that for all $0\le j\le b-1$ one has $V_j=e^\alpha$ conditional on $\{V_j>0\}$. 

\begin{remark}\label{special case}
In  the critical case $h_{V,\nu}=h_\nu(\sigma)$, due to the possible dependence between the coordinates of $V$ a special case  must be distinguished when $Q$ is of $Q_{\log(b)}$-type, namely the case when $\mathbb P(\{\exists\,  0\le j\le b-1, \,V_j=b, \ \forall\, j'\neq j,\, V_{j'}=0\})=1$. Indeed in this case, $h_{\nu,V}=\log(b)$ so $h_{V,\nu}=h_\nu(\sigma)$ implies that $\nu$ equals $\lambda$, and $Q\cdot\nu$ is a Dirac mass almost surely; in particular $\lambda$ is $Q$-regular. Also $\widetilde{\mathbb Q}(S=\infty)=1$ and $\widetilde{\mathbb Q}(L<\infty)=1$. 
\end{remark}

\begin{theorem}\label{main2} Suppose that $\mathbb P(V\neq (1,\ldots,1))>0$ and $\mathbb P(\{\exists\,  0\le j\le b-1, \,V_j=b, \ \forall\, j'\neq j,\, V_{j'}=0\})<1$. 
Let $\nu$ is an ergodic measure satisfying {\bf (G)}; if $Q$ is of $Q_{h_\nu(\sigma)}$-type, suppose that $\nu$ also satisfies \textbf{(}$\mathbf{\widetilde G}$\textbf{)}. If $h_{V,\nu}=h_\nu(\sigma)$, then $\mathbb{E}_\mathbb{P}(Z)=0$, i.e. $\nu$ is $Q$-singular. 
\end{theorem}
\begin{remark}
When $\nu$ is a Bernoulli product measure, Corollary~\ref{cor2} and Theorem~\ref{main2} coincide with the result obtained by Durrett and Liggett~\cite{DuLi83} in their study of  fixed points with finite mean of the so-called smoothing transformation (\cite{DuLi83} also treats the case of solutions with infinite mean). As mentioned in Remark~\ref{rem0}(2), this is also a special case of the result obtained in \cite{Lyons97}.
\end{remark}

\subsection{Applications to the geometry of some random measures}\label{Appli}

Generalising the non-degeneracy criteria from $L^q$-spectra to local dimension has a big advantage in some applications, especially in projection theory when dealing with disintegration of measures. Indeed, in general one works with measures whose dimension properties can be controlled, though their $L^q$-spectrum seems out of reach. We shall show two applications of Theorem~\ref{thm1} in this context.

\subsubsection{\bf Absolute continuity properties of projections of random statistically self-similar measures}\label{sec2.5}

Let $\mathcal{I}=\{f_i\}_{i\in \Lambda}$ be an iterated function system (IFS) consisting of similarities on $\mathbb{R}^d$ ($d\ge 2$) of the form
\[
f_i(x)=r_i O_i x+t_i,\ x\in \mathbb{R}^d
\]
where $r_i\in(0,1)$ is the contraction ratio, $O_i\in SO(d)$ is the rotation component and $t_i\in \mathbb{R}^d$ is the translation. Once such an IFS is given, there exists a unique compact set $E$ in $\mathbb{R}^d$ such that
\[
E=\bigcup_{i\in\Lambda} f_i(E).
\]
The set $E$ is called the self-similar set of $\mathcal{I}$. One says that $\mathcal{I}$ satisfies the open set condition (OSC) if there exists an open set $V$ such that $\cup_{i\in \Lambda} f_i(V)\subset V$ and the sets in the union on the left are disjoint. Furthermore, given a probability vector $p=\{p_i\}_{i\in\Lambda}$ there exists a unique probability measure $\mu_p$ supported on $E$ such that
\[
\mu_p=\sum_{i\in\Lambda} p_i \mu_p\circ f_{i}^{-1}.
\]
This measure is called the self-similar measure associated with $\mathcal I$ and $p$. One way of seeing such a measure is to consider it as the push-forward measure of the Bernoulli measure $\nu_p$ on~$\Sigma$ w.r.t. $p$ via the canonical mapping $\Phi: \Sigma\to E$ given by
\[
\Phi((x_n)_{n=1}^\infty)=\lim_{n\to\infty} f_{x_1}\circ \cdots \circ f_{x_n} (z),
\]
where $z\in\mathbb{R}^d$ can be chosen arbitrarily, that is $\mu_p=\nu_p\circ \Phi^{-1}$.

Let $\mathcal{P}=\{[i]:i\in\Lambda\}$ denote the partition of $\Sigma$ by its first generation cylinders. For a continuous function $\varphi:\Sigma\to \mathbb{R}^d$, denote by $\mathcal{B}_\varphi$ the $\sigma$-field generated by $\varphi^{-1}\mathcal{B}(\mathbb{R}^d)$. Define the conditional information
\[
I_p(\mathcal{P},\varphi)=\sum_{B\in\mathcal{P}} -\chi_B \log \mathbb{E}_{\nu_p}(\chi_B|\mathcal{B}_\varphi)
\]
and the conditional entropy
\[
h_p(\mathcal{P},\varphi)=\mathbb{E}_{\nu_p}(I_p(\mathcal{P},\varphi))=\int_{\Sigma} I_p(\mathcal{P},\varphi)(x) \, \nu_p({\rm d}x).
\]
Finally, denote $h_{\nu_p}(\sigma)=\sum_{i\in\Lambda} -p_i\log p_i$ by $h_p$,  and define $\chi_{p,r}=\sum_{i\in\Lambda} -p_i\log r_i$ the  Lyapunov exponent associated with $\nu_p$ and $\mathcal I$. It is proved in \cite{FengHu2009} that $\mu_p=\nu_p\circ \Phi^{-1}$ is exact-dimensional with dimension
\[
\frac{h_p-h_{p}(\mathcal{P},\Phi)}{\chi_{p,r}}.
\]
Furthermore, we have the following result for its projections and fibres: For $1\le k \le d-1$ denote by $\Pi_{d,k}$ the set of orthogonal projections from $\mathbb{R}^d$ to its $k$-dimensional subspaces. Let $G=\overline{\langle O_i: i\in \Lambda\rangle}$ be the closure in $SO(d)$ of the subgroup generated by the isometries $O_i$. Let $\xi$ denote the normalised Haar measure on $G$. For $\pi\in \Pi_{d,k}$, define
\[
h_{p,\pi}(\mathcal{P},\Phi)=\int_G  h_{p}(\mathcal{P},\pi g \Phi) \,\xi(dg).
\]
It is proved in \cite{FJ14} (the case when $G=\{\mathrm{Id}\}$ is a singleton is proved by Furstenberg in \cite{Fur}) that for $\pi\in\Pi_{d,k}$ and for $\xi$-a.e. $g\in G$,
\begin{itemize}
\item[(i)] $\nu_p\circ(\pi g \Phi)^{-1}$ is exact-dimensional with dimension
\[
\frac{h_p-h_{p,\pi}(\mathcal{P},\Phi)}{\chi_{p,r}};
\]
\item[(ii)] For $\nu_p\circ(\pi g \Phi)^{-1}$-a.e. $y\in \pi g(\mathbb{R}^d)$, $(\mu_p)_{\pi g, y}$ is exact-dimensional with dimension
\[
\frac{h_{p,\pi}(\mathcal{P},\Phi)-h_p(\mathcal{P},\Phi)}{\chi_{p,r}},
\]
where $(\mu_p)_{\pi g, y}$ is the fibre measure obtained by disintegrating $\mu_p$ w.r.t. $\mu_p\circ(\pi g)^{-1}$.
\end{itemize}

Note that $(\mu_p)_{\pi g, y}$ can be interpreted as the push-forward measure through $\Phi$ of the conditional measure $(\nu_p)_{\pi g \Phi,y}$ on $\Sigma$ with respect to the measurable partition $\{[y]_\pi=(\pi g \Phi)^{-1}(y): y\in \pi g \Phi(\Sigma)\}$. Moreover, for $\nu_p$-a.e. $x\in\Sigma$, the measure $(\nu_p)_{\pi g \Phi,\pi g\Phi(x)}$ is exact-dimensional with dimension 
\[
\dim_H((\nu_p)_{\pi g \Phi,\pi g\Phi(x)})=h_{p,\pi}(\mathcal{P},\Phi).
\]
One can see this from the proof of Theorem 3.2 (iii) in \cite{FJ14}: while computing the local dimension of $(\mu_p)_{\pi g, y}=(\nu_p)_{\pi g \Phi,y}\circ \Phi^{-1}$, one uses the pull-back balls $B_\Phi(x,n):=\Phi^{-1}(B(\Phi(x), r_{x,n}))$ on the symbolic space, where $x\in \Sigma$ and $r_{x,n}$ is a suitably chosen decreasing sequence of radii. As the canonical mapping $\Phi:\Sigma\to \mathbb{R}^d$ is not necessarily one-to-one, the pull-back balls $B_\Phi(x,n)$ may be quite different comparing to cylinders, this is reflected in the diminution of the entropy by the amount $h_p(\mathcal{P},\Phi)$ in the dimension formula. But when computing the local dimension of $(\nu_p)_{\pi g \Phi,y}$, one can directly use the cylinders, or in other words, one can use the canonical mapping $\widetilde\Phi$ of arbitrary IFS with strong separation condition so that $B_{\widetilde\Phi}(x,n)$ is exactly the cylinder $[x|_n]$. This results in that the entropy $h_p(\mathcal{P},\widetilde\Phi)=0$ as $\mathcal{P}\subset \mathcal{B}=\mathcal{B}_{\widetilde\Phi}$. 

For  fibre measures, either $(\mu_p)_{\pi g, y}$ or $(\nu_p)_{\pi g \Phi,x}$, it is really hard to say anything beyond their dimensions, or their multifractal spectrum. But our new theorem makes it possible to determine whether these fibre measures are $Q$-regular, and their $Q$-regularity has some nice implications.

We shall only consider the case where {\bf (M1)} holds. The case of {\bf (M2)} would be similar but  it requires further works to determine that all typical fibre measures have the same entropy after performing the $Q$-action. Recall that $h_X=\mathbb{E}(X\log X)$. 

\begin{theorem}\label{Application2}
For arbitrary $\pi\in\Pi_{d,k}$, if $h_X<h_{p,\pi}(\mathcal{P},\Phi)$, then, with probability~$1$, conditional on $Q\cdot \nu_p\neq 0$, for $\xi$-a.e. $g\in G$, $(Q\cdot \nu_p) \circ (\pi g \Phi)^{-1}$ is absolutely continuous with respect to $\nu_p\circ(\pi g \Phi)^{-1}$ with the Radon-Nikodym derivative
\[
\pi g\Phi(\Sigma)\ni y \to \|Q\cdot (\nu_p)_{\pi g \Phi,y}\|.
\]
Furthermore, if the IFS $\mathcal{I}$ satisfies the OSC, then  $(Q\cdot \nu_p)\circ \Phi^{-1}$ is exact-dimensional with dimension
\[
\frac{h_p-h_X}{\chi_{p,r}}.
\]
\end{theorem}

The claim regarding the absolute continuity of the projections with respect to their expectations is an extension of  \cite[Theorem 3.1(1)]{BF2018} to arbitrary self-similar IFSs and to arbitrary projections (composed with Haar measure almost every rotations). Also, in the OSC case, the fact that the Hausdorff dimensions of the measure $(Q\cdot \nu_p)\circ \Phi^{-1}$ is obtained without additional moment assumption on $X$ is an improvement with respect to \cite[Theorem 3.1(i)]{FJ14}. Note that in this case, since $h_X<h_{p,\pi}(\mathcal{P},\Phi)$, we have
\[
\dim((Q\cdot \nu_p)\circ \Phi^{-1})=\frac{h_p-h_X}{\chi_{p,r}}>\frac{h_p-h_{p,\pi}(\mathcal{P},\Phi)}{\chi_{p,r}}=\dim(\nu_p\circ(\pi g \Phi)^{-1}).
\]
Thus, we have the same phenomenon as  in  \cite[Theorem 3.1(1)]{BF2018}:  if the dimension of the random statistically self-similar measure $(Q\cdot \nu_p)\circ \Phi^{-1}$ is larger than the dimension of the projection  of its expectation $\nu_p\circ\Phi^{-1}$ through $\pi g$, then almost surely $(Q\cdot \nu_p) \circ (\pi g \Phi)^{-1}$ is absolutely continuous with respect to $\nu_p\circ(\pi g \Phi)^{-1}$.

\subsubsection{\bf Hausdorff dimension of some random measures on Bedford-McMullen carpets}\label{sec2.6}

The results of Section~\ref{secMMC} can also be used to compute the Hausdorff dimension of some Mandelbrot measures on Bedford-McMullen carpets. This approach, alternative to that used in~\cite{BF2020}, makes it possible to remove a moment assumption required in~\cite{BF2020} to get  the Hausdorff dimension. In fact, we are going to study a wider class of random measures obtained as limit of the action of multiplicative cascades  on ergodic measures.  

We fix two integers $b_1>b_2$, $\Lambda_1=\{1,\ldots,b_1\}$,  $\Lambda_2=\{1,\ldots,b_2\}$, as well as the product symbolic space $\Lambda_1^{\mathbb Z_+}\times\Lambda_2^{\mathbb Z_+}$, $\pi_2$ the canonical projection from $\Lambda_1^{\mathbb Z_+}\times\Lambda_2^{\mathbb Z_+}$ to $\Lambda_2^{\mathbb Z_+}$, and $\sigma=(\sigma_1,\sigma_2)$ the shift operation on the product space $\Lambda_1^{\mathbb Z_+}\times\Lambda_2^{\mathbb Z_+}$. Endow the space $\Lambda_1^{\mathbb Z_+}\times\Lambda_2^{\mathbb Z_+}$ with the metric 
$$
d_{(b_1,b_2)}=\left ((x_n,y_n)_{n=0}^\infty,(x'_n,y'_n)_{n=0}^\infty\right )=\max \left (b_1^{-\inf\{n\ge 0:x_n\neq x'_n\}},b_2^{-\inf\{n\ge 0:y_n\neq y'_n\}}\right ).
$$

Fix a subset $\Lambda$ of $\Lambda_1\times\Lambda_2$ of cardinality at least 2, and let $\nu$ be an ergodic measure on $(\Lambda_1^{\mathbb Z_+}\times\Lambda_2^{\mathbb Z_+},\sigma)$, supported on $\Lambda^{\mathbb Z_+}$. We identify $\nu$ with its restriction to $\Lambda^{\mathbb Z_+}$, and denote by ${\pi_2}_*\nu$ its projection to $\Sigma_2$.

It is know that $\nu$ disintegrates as ${\pi_2}_*\nu({\rm d}y) \nu^y({\rm d}x)$,  where for ${\pi_2}_*\nu$-almost every $y$, the conditional measure $\nu^y({\rm {\rm d}x})$ is supported on $\Lambda\cap (\Lambda_1^{\mathbb Z_+}\times\{y\})$, and  the relativized Shannon-Breiman-McMillan theorem (\cite[Lemma 4.1]{Bog1999})  states that $\nu^y$ is exact dimensional with dimension $h_\nu(\sigma)-h_{{\pi_2}_*\nu}(\sigma_2)$ with respect to the metric induced on $\Lambda_1^{\mathbb Z_+}\times\{y\}$ by the metric $d$.

\begin{theorem}\label{Application1} Consider a Mandelbrot multiplicative cascade $(X_u)_{u\in \Lambda^*}$ satisfying {\bf (M1)} as in Section~\ref{independent cascade}. 
Suppose that $h_X<h_\nu(\sigma)-h_{{\pi_2}_*\nu}(\sigma_2)$. Then $\nu$ is $Q$-regular, and with probability~1, conditional on $Q\cdot\nu\neq 0$, ${\pi_2}_* (Q\cdot\nu)$ is absolutely continuous with respect to ${\pi_2}_*\nu$, and $Q\cdot\nu$ is exact dimensional with 
\begin{align*}
\dim(Q\cdot\nu)&=\frac{1}{\log b_1} (h_\nu(\sigma)-h_X)+\Big (\frac{1}{\log b_2} - \frac{1}{\log b_1}\Big )h_{{\pi_2}_*\nu}(\sigma_2)
\end{align*}
with respect to the metric $d_{(b_1,b_2)}$.
\end{theorem}

When $\nu$ is a Bernoulli product measure, it is easy to see that $Q\cdot\nu$ is a special case of statistically self-affine Mandelbrot measure as considered in \cite{BF2020}. The previous theorem improves the result obtained in \cite{BF2020} on the computation of $\dim_H (Q\cdot\nu)$ by removing the moment assumption $\mathbb{E}(X^q)<\infty$ for some $q>1$. However, under this moment assumption, \cite{BF2020}  computes $\dim_H (Q\cdot\nu)$ in all the cases where $Q\cdot\nu$ is non degenerate, with $\nu$ a Bernoulli product measure.

\bigskip

{\bf Organisation of the rest of the paper.}  Theorem \ref{thm1} and Theorem~\ref{cor1} are proved in Section~\ref{Proofthm1}, while the proofs of Corollaries~\ref{cor-2.11} and~\ref{cor2} are given  in Section~\ref{pfcor2}, and that of Theorem~\ref{main2} (which implies Theorem~\ref{thm2}) in Section~\ref{pfmain2}. Finally, Theorems~\ref{Application2}~and~\ref{Application1} are proved in Section~\ref{PFAPPl}.

\section{Proofs of Theorem \ref{thm1} and Theorem~\ref{cor1}}\label{Proofthm1}

Since Theorem~\ref{cor1} implies Theorem \ref{thm1}(a) and (b), we first prove Theorem~\ref{cor1} and then Theorem \ref{thm1}(c). 

\begin{proof}[Proof of Theorem~\ref{cor1}] 
For $n\ge 1$ define the random variable $U_n$ on $\widetilde{\Omega}$ by
\[
U_n(\omega,x)=\log X_{x|_n}(\omega).
\]
Given any finite sequence of non-negative measurable functions $f_{1},\ldots,f_{k}$ and any increasing sequence of integers $1\le n_1<n_2<\cdots<n_k$, we have
\begin{align*}
\mathbb{E}_{\widetilde{\mathbb{Q}}}(f_1(U_{n_1})f_2(U_{n_2})\cdots f_k(U_{n_k}))=&\mathbb{E}_{\mathbb{P}}\Big(\int_\Sigma \prod_{i=1}^{n_k}X_{x|_i}\prod_{j=1}^k f_j(\log X_{x|_{n_j}}) \, \nu({\rm d}x) \Big)\\
=&\int_\Sigma \mathbb{E}_{\mathbb{P}}\Big( \prod_{i=1}^{n_k}X_{x|_i}\prod_{j=1}^k f_j(\log X_{x|_{n_j}})\Big)\,  \nu({\rm d}x) \\
=& \int_\Sigma  \prod_{j=1}^{k}\mathbb{E}_{\mathbb{P}}\Big(X_{x|_{n_j}} f_j(\log X_{x|_{n_j}})\Big)\,  \nu({\rm d}x) \\
=&\int_\Sigma  \prod_{j=1}^{k}\mathbb{E}_{\mathbb{P}}\Big(Xf_j(\log X)\Big)\,  \nu({\rm d}x)= \prod_{j=1}^{k}\mathbb{E}_{\mathbb{P}}\Big(Xf_j(\log X)\Big).
\end{align*}
This implies that the random variables $U_n$, $n\ge 1$, are  i.i.d. under $\widetilde{\mathbb{Q}}$, and their common law is given by
\[
\mathbb{E}_{\widetilde{\mathbb{Q}}}(f(U_{n}))=\mathbb{E}_{\mathbb{P}}(Xf(\log X))
\]
for all non-negative measurable functions $f$. Consequently, since $x\log x \ge -e^{-1}$ on $\mathbb{R}_+$, the negative part of $U_1$ is $\widetilde{\mathbb{Q}}$-integrable, and
\[
\mathbb{E}_{\widetilde{\mathbb{Q}}}(U_1)=\mathbb{E}_{\mathbb{P}}(X\log X)=h_X \in [0,\infty].
\]
Then, using the strong law of large numbers we obtain that for $\widetilde{\mathbb{Q}}$-almost every $(\omega,x)\in\widetilde{\Omega}$,
\begin{equation}\label{un}
\lim_{n\to\infty}\frac{1}{n}(U_1+\cdots+U_n)=h_X.
\end{equation}

By \eqref{un}, for $\widetilde{\mathbb{Q}}$-almost every $(\omega,x)\in\Omega\times \Sigma_{h_X,\nu}^+$, we have
\[
\limsup_{n\to \infty} \frac{1}{n} \log [Y_{x|_{n}}(\omega)\nu([x|_{n}])]<0.
\]
This implies that $S(\omega,x)<\infty$ for $\widetilde{\mathbb{Q}}$-almost every $(\omega,x)\in\Omega\times \Sigma_{h_X,\nu}^+$. Note that $\widetilde{\mathbb{Q}}(\Omega\times \Sigma_{h_X,\nu}^+)=\nu(\Sigma_{h_X,\nu}^+)$. Hence, by Theorem \ref{main1}, if $\nu(\Sigma_{h_X,\nu}^+)>0$ then $\mathbb{E}_{P}(Z)>0$ and if $\nu(\Sigma_{h_X,\nu}^+)=1$, then $\mathbb{E}_{P}(Z)=1$.

By \eqref{un} we also have for  $\widetilde{\mathbb{Q}}$-almost every $(\omega,x)\in\Omega\times \Sigma_{h_X,\nu}^-$,
\[
\limsup_{n\to \infty} \frac{1}{n} \log Y_{x|_{n}}(\omega)\nu([x|_{n}])>0.
\]
This implies that $L(\omega,x)=\infty$ for $\widetilde{\mathbb{Q}}$-almost every $(\omega,x)\in\Omega\times \Sigma_{h_X,\nu}^-$. Note that $\widetilde{\mathbb{Q}}(\Omega\times \Sigma_{h_X,\nu}^-)=\nu(\Sigma_{h_X,\nu}^-)$. Hence, by Theorem \ref{main1}, if $\nu(\Sigma_{h_X,\nu}^-)>0$ then $\mathbb{E}_{P}(Z)<1$ and if $\nu(\Sigma_{h_X,\nu}^-)=1$ then $\mathbb{E}_{P}(Z)=0$.
\end{proof}

\begin{proof}[Proof of Theorem \ref{thm1}(c)]  Once Theorem \ref{thm1}(a) is known, the first inequality is proven  by using the same approach as Kahane in \cite{Kahane1987} in the case that $\nu$ is the uniform measure and Fan to get \cite[Theorem B(a)]{Fan} when $\nu$ is an ergodic $\sigma$-invariant Markov measure. It consists in  using the composition principle of multiplicative chaos actions to prove that  conditional on $Q\cdot\nu\neq 0$, this measure is  $Q_\alpha$ regular for all $\alpha \in (0, \underline \dim(\nu)- h_X)$ (recall that the multiplicative chaos $Q_\alpha$ was defined just after Theorem~\ref{cor1}, and see \cite{WW95} for a proof of the composition principle). 

\medskip 

For the second inequality, denote $\overline \dim_P(\nu)$ by $\beta$, and $\beta-\mathbb{E}_{\mathbb P}(X\log(X))$ by $\beta_X$. Fix $\epsilon>0$. For $n\in\mathbb N$ set 
$
E_{\epsilon,N}=\{x\in\Sigma: \forall\, n\ge N,\, \nu([x|_n])\ge e^{-n (\beta+\epsilon)}\}$. Then let  $\widetilde E_{\epsilon,N}$ be a compact subset of $E_{\epsilon,N}$ of measure larger than $(1-2^{-N})\nu(E_{\epsilon,N})$. Then, for any integer $p\ge N$ define $F_{\epsilon,N,p}=\{x\in \widetilde E_{\epsilon,N}: Q\cdot \nu([x|_p])< e^{-p (\beta_X+2\epsilon)}\}.
$
For any $q\in (0,1)$, we have, using Markov inequality
$$
Q\cdot \nu (F_{\epsilon,N,p})\le \sum_{\substack{u\in\Lambda^p\\\nu([u])\ge e^{-p (\beta+\epsilon)}}} \mathbf{1}_{Q\cdot \nu([u])>0}\big ((e^{-p (\beta_X+2\epsilon)}(Q\cdot \nu([u]))^{-1}\big )^{(1-q)} \,Q\cdot \nu([u]).
$$
Taking expectation, and using that  
\begin{align*}
\mathbb{E}((Q\cdot \nu([u])^q)&\le \liminf_{n\to\infty} \mathbb{E}((Q_n\cdot \nu([u])^q)\\
&= \liminf_{n\to\infty} \mathbb{E}(Y_p^q  (Q^u_{n-p}\cdot \nu([u]))^q)\\
&\le \mathbb{E}(X^q)^p \liminf_{n\to\infty} (\mathbb{E}(Q_{n-p}^u\cdot \nu([u]))^q
=\mathbb{E}(X^q)^p\nu([u])^q,
\end{align*} we get
\begin{align*}
\mathbb E(Q\cdot \nu (F_{\epsilon,N,p}))&\le \sum_{\substack{u\in\Lambda^p\\\nu([u])\ge e^{-p (\beta+\epsilon)}}} \mathbb{E}(X^q)^pe^{-p (\beta_X+2\epsilon) (1-q)} \nu([u])^{q-1} \nu([u])\\
&\le \mathbb{E}(X^q)^pe^{-p (\beta_X+2\epsilon) (1-q)+p(\beta+\epsilon)(1-q)}=e^{-p(1-q)(\beta_X(q)+\epsilon)},
\end{align*}
where $\beta_X(q)=\beta_X+\frac{\log \mathbb{E}(X^q)}{q-1}-\beta=\frac{\log \mathbb{E}(X^q)}{q-1}-\mathbb{E}_{\mathbb P}(X\log(X))=o(1)$ as $q\to 1$. Consequently we can choose $q\in (0,1)$ such that  $\mathbb E(Q\cdot \nu (F_{\epsilon,N,p}))\le e^{-p(1-q)\epsilon/2}$. By the Borel-Cantelli lemma, this implies that with probability 1, conditional on $Q\cdot \nu\neq 0$, for $Q\cdot \nu$-almost every $x\in  \widetilde E_{\epsilon,N}$, $\overline\dim_{\rm{loc}}(Q\cdot \nu,x)\le \beta_X+\epsilon$. However, by definition of $\beta$, we have $\lim_{N\to\infty} \nu(E_{\epsilon,N})=1$, so $\lim_{N\to\infty} \nu(\widetilde E_{\epsilon,N})=1$. Setting $G_\epsilon= \bigcap_{N\in\mathbb N^*} \widetilde E_{\epsilon,N}^c$, we have $Q\cdot \nu(G_\epsilon)\le Q\cdot \nu(\widetilde E_{\epsilon,N}^c)\le \liminf_{n\to\infty} Q_n\cdot \nu(\widetilde E_{\epsilon,N}^c)$ since $\widetilde E_{\epsilon,N}^c$ is open. Hence by Fatou's Lemma, $\mathbb{E}(Q\cdot \nu(G_\epsilon))\le \nu(\widetilde E_{\epsilon,N}^c)$ for all $N\ge 1$ and finally, $Q\cdot \nu(G_\epsilon)=0$ almost surely. 

Since the previous properties  hold for all $\epsilon>0$, we can get that, with probability 1, conditional on $Q\cdot\nu\neq 0$, $\overline\dim_{\rm{loc}}(Q\cdot \nu,x)\le \beta_X$ for $Q\cdot \nu$-almost every $x\in \Sigma\setminus \bigcup_{j\ge 1} G_{1/j}$, which is a set of full $Q\cdot\nu$-measure.  
\end{proof}

\begin{remark}
If we assume that $\mathbb E (X^q)<\infty$ for some $q>1$, the above approach can also be used to get the first inequality in $(c)$.   
\end{remark}

\section{Proofs of Corollaries~\ref{cor-2.11} and~\ref{cor2}}\label{pfcor2}

\begin{proof}[Proof of Corollary~\ref{cor-2.11}]
First, by definition, for any $n\ge 1$,  setting $\mu_n:=Q_n \cdot \int_{\overline {\mathcal E} (h_X)} \nu\, \rho (\mathrm{d}\nu)$, for any non-negative continuous function $f$ defined~on~$\Sigma$, one has, by using Fubini-Tonelli's Theorem,
$$
\int_\Sigma f\, {\rm d}\mu_n=\int_{\overline {\mathcal E} (h_X)}\left ( \int_\Sigma f\, {\rm d}(Q_n\cdot\nu)\right)\, \rho (\mathrm{d}\nu).
$$ 
We know from Theorem~\ref{thm1} that  for all $\nu\in \overline {\mathcal E} (h_X)$, $\mathbb P$-almost surely, $\nu$ is $Q$-regular. Consequently, an application of the Fubini-Tonelli theorem yields that this holds $\mathbb P$-almost surely, for $\rho$-almost every $\nu\in \overline {\mathcal E} (h_X)$. This implies that $ \mathbb{E}(\int_\Sigma f\, Q\cdot {\rm d}\nu)=\int_\Sigma f{\rm d}\nu$ for $\rho$-almost every $\nu\in \overline {\mathcal E} (h_X)$. Thus, 
\begin{equation}\label{expeq}
\mathbb{E}\left (\int_{\overline {\mathcal E} (h_X)}\int_\Sigma f\, {\rm d}(\lim_{n\to\infty} Q_n\cdot \nu) \, \rho (\mathrm{d}\nu)\right)=\int_\Sigma f\, {\rm d}\nu=\mathbb {E}\left (\int_{\overline {\mathcal E} (h_X)} \int_\Sigma f\, {\rm d}(Q_n\cdot \nu) \, \rho (\mathrm{d}\nu)\right).
\end{equation}
On the other hand, Fatou's lemma implies \begin{align*}
\mathbb{E}\left (\int_{\overline {\mathcal E} (h_X)}\int_\Sigma f\,{\rm d} (\lim_{n\to\infty} Q_n\cdot \nu) \, \rho (\mathrm{d}\nu)\right)&\le \mathbb {E}\left (\lim_{n\to\infty}\int_{\overline {\mathcal E} (h_X)} \int_\Sigma f\, {\rm d}(Q_n\cdot \nu) \, \rho (\mathrm{d}\nu)\right)\\
&\le 
\lim_{n\to\infty}\mathbb {E}\left (\int_{\overline {\mathcal E} (h_X)} \int_\Sigma f\, {\rm d}(Q_n\cdot \nu) \, \rho (\mathrm{d}\nu)\right).
\end{align*}
From this and \eqref{expeq} we deduce that the non-negative martingale $(\int_\Sigma f {\rm d} \mu_n)_{n\ge 1}$ is uniformly integrable, hence the measure $\int_{\overline {\mathcal E} (h_X)} \nu\, \rho (\mathrm{d}\nu)$ is $Q$-regular. Moreover, since 
\begin{equation}\label{expeq2}
\int_{\overline {\mathcal E} (h_X)}\int_\Sigma f\,{\rm d} (\lim_{n\to\infty} Q_n\cdot \nu)\,  \rho (\mathrm{d}\nu)\le \lim_{n\to\infty}\int_{\overline {\mathcal E} (h_X)} \int_\Sigma f\, {\rm d}(Q_n\cdot \nu) \,  \rho (\mathrm{d}\nu)
\end{equation}
$\mathbb{P}$-almost surely (by Fatou's lemma again), we conclude that \eqref{expeq2}  is an equality $\mathbb P$-almost surely. This yields \eqref{regular-inv}. 

\medskip

Let us next consider $\mu=\displaystyle\int_{\underline {\mathcal E} (h_X)} \nu\, \rho (\mathrm{d}\nu)$. This time, using \eqref{expeq2} and the fact that each $\nu\in \underline {\mathcal E} (h_X)$ is $Q$-singular yields the $Q$-singularity of $\mu$. 
\end{proof}

\begin{proof}[Proof of Corollary~\ref{cor2}] It is almost the same as the proof of Theorem \ref{cor1}. The only difference is that the sequence
\[
\{U_n(\omega,x)=\log X_{x|_n}(\omega)\}_{n\ge 1}
\]
is now, instead of i.i.d., ergodic under $\widetilde{\mathbb{Q}}$. One can check this easily by using the ergodicity of $\nu$ and the independence of the random vectors in different generations. We shall prove a more general case in Lemma \ref{hatT}. Hence, by Birkhoff ergodic theorem, $\widetilde{\mathbb{Q}}$-almost surely,
\[
\lim_{n\to\infty}\frac{1}{n}(U_1+\cdots+U_n)=\mathbb{E}_{\widetilde{\mathbb{Q}}}(U_1)=\sum_{j\in\Lambda} \mathbb{E}(V_j\log V_j)\nu([j])=h_{V,\nu}.
\]
On the other hand, by Shannon-McMillan-Breiman theorem, $\nu$-almost surely,
\[
\lim_{n\to\infty} \frac{-\log \nu([x|_n])}{n}=h_\nu(\sigma).
\]
Then the conclusion easily follows from Theorem \ref{main1}.
\end{proof}

\section{Proof of Theorem~\ref{main2}}\label{pfmain2}

Let us begin with the easy part of the result, that is when $Q$ is of $Q_\alpha$-type and $\nu$ satisfies \textbf{(}$\mathbf{\widetilde G}$\textbf{)}. Note that  to be in the critical case we must have $\alpha=h_\nu(\sigma)$.

If $\nu\neq \lambda$, \eqref{Gtilda} directly implies that $L(\omega,x)=\infty$ $\widetilde{\mathbb Q}$-a.e.

If $\nu=\lambda$, then $h_{V,\lambda}=h_\nu(\lambda)=\alpha=\log(b)$, which means that $\mathbb{E}\sum_{j=0}^{b-1} \mathbf{1}_{\{V_j>0\}}=1$. Since we assumed that $\mathbb P(\{\exists\,  0\le j\le b-1, \,V_j=b, \ \forall\, j'\neq j,\, V_{j'}=0\})<1$, this implies that the topological support of $Q\cdot\nu$ is contained in a degenerate critical Galton-Watson tree, i.e. it is empty, hence $Q\cdot\nu=0$ almost surely.

\medskip

To deal with the case where $Q$ is not of $Q_{h_\nu(\sigma)}$-type, we first define a larger product probability space associated with the natural extension of~$\nu$ to $\Lambda^{\mathbb Z}$. 

\subsection{Natural extension of $\nu$ and associated cascade probability space}  
 The bilateral symbolic space $\Lambda^{\mathbb Z}$ is still denoted by $\Sigma$, and  the left shift operation on $\Lambda^{\mathbb Z}$ is denote by $\sigma$. 
 
 For $x=\cdots x_{-1}x_0x_1\cdots\in\Sigma$ and $n\in\mathbb Z_+$, define  $x|_0=\emptyset$,
 \[
x|_{n}=x_1x_2\cdots x_{n}
\]
if $n\ge 1$, and 
\[
x|_n=x_{n+1}x_{n+2}\cdots x_{0}
\]
if $n\le -1$. For $u,v\in\Lambda^*$, set 
\[
[u,v]=\{x \in \Sigma: x|_{-|u|}=u, x|_{|v|}=v\},
\]
the cylinder rooted at $uv$. We shall use the convention that $[\emptyset,v]=[v]$ and $[u,\emptyset]=[u]^-$. 

\smallskip

The $\sigma$-field generated by these bilateral cylinders is denoted by $\mathcal B$. It is standard that the $\sigma$-ergodic measure $\nu$ possesses  a natural extension to the bilateral symbolic space $(\Sigma,\mathcal B)$, which is still $\sigma$-ergodic (see e.g. \cite[Ch. 6]{Breiman1968}).

\medskip

Note that property {\bf (G)} still holds for this natural extension if we redefine $f$ by $f(x):=f(x_0x_1\cdots)$ and $C(x):=C(x_0x_1\cdots)$.

For convenience we may define the probability space more precisely: 
\[
(\Omega,\mathcal{A},\mathbb{P})=\bigotimes_{u\in\Lambda^*}(\Omega_u,\mathcal{A}_u,\mathbb{P}_u),
\]
where for each $u\in\Lambda^*$ the probability space $(\Omega_u,\mathcal{A}_u,\mathbb{P}_u)$ are the same one on which the random vector $V$ is defined. For $u\in\Lambda^*$ define the projection
\[
\theta_u:\Omega\ni \omega=(\omega_v)_{v\in\Lambda^*}\to \omega_u\in \Omega_u.
\]
Then $\{(X_{uj})_{j\in\Lambda}=V\circ \theta_u:u\in\Lambda^*\}$ forms an i.i.d. sequence of random vectors on $(\Omega,\mathcal{A},\mathbb{P})$.

For $n\ge 1$ define the $\sigma$-field
\[
\widetilde{\mathcal{F}}_{n}=\sigma\{\chi_{[u,v]}X_{v}:u,v\in\Lambda^*,\ |v|\le n\}
\]
and set $\widetilde{\mathcal{F}}=\sigma(\bigcup_{n\ge 1}\widetilde{\mathcal{F}}_{n})$.

\subsection{Proof of the theorem when $Q$ is not of $Q_{h_\nu(\sigma)}$-type}

The fact that $Q$ is not of $Q_{h_\nu(\sigma)}$-type will be used to get Lemma~\ref{stone} below.

\medskip

First we need the following result on the recurrence of Birkhoff sums with zero mean essentially due to Dekking \cite{Dek}: Let $(X,\mathcal{B},T,\mu)$ be an ergodic dynamical system. Let $f:X\mapsto \mathbb{R}$ be a measurable function. For $n\ge 1$ denote by $S_nf(x)=\sum_{k=0}^{n-1} f\circ T^k(x)$ the $n$th-Birkhoff sum of $f$.  For $b>0$ and $x\in X$ define
\[
\tau_b(x)=\inf\{n\ge1: S_nf(x)\in [-b,b]\}.
\]
\begin{proposition}\label{rec}
Suppose that 
\[
\lim_{n\to\infty} \frac{1}{n}S_nf(x)=0 \quad\text{for $\mu$-a.e. $x\in X$.}
\]
Then for all $b>0$, $\tau_{b}(x)<\infty$ for $\mu$-a.e. $x\in X$. 
\end{proposition}
\begin{remark}
Note here that we do not require $f$ to be in $L^1(\mu)$. By ergodicity the condition is  certainly satisfied when $f\in L^1(\mu)$ with $\int_X f \, d\mu=0$.
\end{remark}

By {\bf(G)},  for $\nu$-a.e. $x\in\Sigma$ we have
\[
-\log C(x) \le \log \nu([x|_{n}])-\sum_{k=0}^{n-1} f(\sigma^k x)\le \log C(x).
\]
Let $F=h_\nu(\sigma)+f$. Then, due to the Shannon-McMillan-Breiman theorem applied to $\nu$, for $\nu$-a.e. $x\in \Sigma$,
\[
\lim_{n\to\infty} \frac{1}{n} \sum_{k=0}^{n-1} F(\sigma^k x)=0.
\]
Consider the skew product $\varphi$ on $\Sigma\times\mathbb{R}$:
\[
\varphi(x,u)=(\sigma x, u+F(x)).
\]
It is easy to see that $\varphi$ is invertible and $\mu=\nu\times\mathcal{L}$ is $\varphi$-invariant, where $\mathcal{L}$ is the Lebesgue measure on $\mathbb{R}$. The issue here is that $\mathcal{L}$ is an infinite measure. Therefore we need to consider an induced dynamical system.

Let $A=\Sigma \times[-1/2,1/2]$, $\mathcal{B}_A=\mathcal{B}\otimes \mathcal{B}([-1/2,1/2])$, $\mu_A=\nu\times \mathcal{L}|_{[-1/2,1/2]}$ and
\[
\tau_A(x,t)=\inf\{n\ge 1: \varphi^n(x,t)\in A\} \text{ for } (x,t)\in A.
\]
By Proposition \ref{rec}, for all $0<\epsilon<1/2$ and $t\in [-\epsilon,\epsilon]$ we have
\begin{align*}
\tau_A(x,t)=&\inf\{n\ge 1: t+S_nF(x)\in[-1/2,1/2]\}\\
 \le& \inf\{n\ge 1: S_nF(x)\in[-1/2+\epsilon,1/2-\epsilon]\}\\
<&\infty
\end{align*}
for $\nu$-a.e. $x\in\Sigma$. This implies that for $\mu_A$-a.e. $(x,t)\in A$, $\tau_A(x,t)<\infty$. Define
\[
\varphi_A(x,t)=\varphi^{\tau_A(x,t)}(x,t),\ (x,t)\in A,
\]
with the convention that $\varphi_A(x,t)=(x,t)$ if $\tau_A(x,t)=\infty$.

It is easy to check that the measure $\mu_A$ is $\varphi_A$-invariant. It is not necessarily ergodic; this depends on the function $F$. But we can consider its ergodic decomposition and write it as $\mu_A=\int_{\mathcal E(\varphi_A)}\xi\, \rho({\mathrm d} \xi)$, where $\mathcal E(\varphi_A)$ stands for the set of ergodic Borel probability measures on $(A, \mathcal{B}_A,\varphi_A)$. For the simplicity of notations, in the following we pick an ergodic measure $\xi$ according to the probability measure $\rho$. 
Now consider the product space
\[
\widehat{\Omega}=\Omega\times A, \ \widehat{\mathcal{B}}=\widetilde{\mathcal{F}}\otimes \mathcal{B}(\mathbb{R})|_{\widehat{\Omega}}
\]
and define a probability measure $\widehat{\mathbb{Q}}$ on $\widehat{\Omega}$ similar to $\widetilde{\mathbb{Q}}$ by
\[
\int_{\widehat{\Omega}} f(\omega,x,t) \, \widehat{\mathbb{Q}}({\rm d}(\omega,x,t))=\int_{\widehat{\Omega}} f(\omega,x,t) Q_{n}(\omega,x) \, \mathbb{P}({\rm d}\omega) \xi({\rm d}(x,t)),
\]
for all bounded measurable functions $f$ on $(\widehat{\Omega},\widehat{\mathcal{B}})$ that are $\widetilde{\mathcal{F}}_{n}$-measurable when restricted to $\widetilde{\Omega}$. For $v\in\Lambda^*$ define $\eta_v:\Omega\to\Omega$ by
\[
\eta_v((\omega_u)_{u\in\Lambda})=(\omega_{vu})_{u\in\Lambda^*}.
\]
For $(\omega,x,t)\in \widehat{\Omega}$ define
\begin{align*}
T_\varphi(\omega,x,t)&=(\eta_{x|_1}\omega, \varphi(x,t))\\
\text{and}\quad 
\widehat{T}(\omega,x,t)&=T_\varphi^{\tau_A(x,t)}(\omega,x,t)=(\eta_{x|_{\tau_A(x,t)}}\omega,\varphi^{\tau_A(x,t)}(x,t)),
\end{align*}
with the convention that $\widehat{T}(\omega,x,t)=(\omega,x,t)$ if $\tau_A(x,t)=\infty$. 

The following two lemmas will be proved in Sections~\ref{pflem1}~and~\ref{pflem2} respectively. 
\begin{lemma}\label{hatT}
$(\widehat{\Omega},\widehat{\mathcal{B}},\widehat{T},\widehat{\mathbb{Q}})$ is ergodic.
\end{lemma}

Define
\[
W(\omega,x,t)=\sum_{k=1}^{\tau_A(x,t)} \left (\log X_{x|_k}(\omega)-h_\nu(\sigma)\right ).
\]
We have

\begin{lemma}\label{stone} If $Q$ is not of $Q_{h_{\nu}(\sigma)}$-type, then
\[
\widehat{\mathbb{Q}}\left(-\infty =\liminf_{n\to \infty} \sum_{k=0}^{n-1} W\circ  \widehat{T}^{k}<\limsup_{n\to \infty} \sum_{k=0}^{n-1} W\circ  \widehat{T}^{k}=\infty\right)=1.
\]
\end{lemma}

We can now end the proof of the theorem. Define 
\[
V(x,t)=\sum_{k=0}^{\tau_A(x,t)-1} F\circ \sigma^k(x)
\]
and for $n\ge 1$, define 
\[
N_n(x,t)=\tau_A(x,t)+\tau_A(\varphi_A(x,t))+\cdots+\tau_A(\varphi^{n-1}_A(x,t)).
\]
Then for $n\ge 1$,
\[
\log Y_{x|_{N_n(x,t)}}(\omega)+\log \nu[x|_{N_n(x,t)}]\ge \sum_{k=0}^{n-1} W\circ  \widehat{T}^{k}(\omega,x,t)+\sum_{k=0}^{n-1} V\circ \varphi_A^k(x,t)-C(x).
\]
Since, by the definition of $\tau_A$, for $\mu_A$-a.e. $(x,t)$, for all $n\ge 1$,
\[
\sum_{k=0}^{n-1} V\circ \varphi_A^k(x,t) \in [-t-1/2,-t+1/2].
\]
Together with Lemma \ref{stone} we deduce that for $\widetilde{\mathbb{Q}}$-a.s.,
\[
L(\omega,x)=\infty.
\]
By Theorem \ref{main1} we deduce that $\mathbb{E}_{\mathbb{P}}(Z)=0$. \qed

\subsection{Proof of Lemma \ref{hatT}}\label{pflem1}

First we show that $\widehat{\mathbb{Q}}$ is $\widehat{T}$-invariant. Consider a Borel set $B$ in $[-1/2,1/2]$, a cylinder $[u,v]$ in $\Sigma$ with $|v|=k\ge 1$ and a set
\[
B'\in\sigma\{X_{w}: w=v_1\cdots v_j \text{ for some } j=1,\ldots,k\}.
\]
Then $f=\chi_{B'\times [u,v]\times B}$ is an elementary function, and it is enough to show that
\[
\int_{\widehat{\Omega}} f(\omega,x,t) \,\widehat{\mathbb{Q}}\circ \widehat{T}^{-1}({\rm d}(\omega,x,t))=\int_{\widehat{\Omega}} f(\omega,x,t) \,\widehat{\mathbb{Q}}({\rm d}(\omega,x,t)).
\]
For $n\ge 1$, define $A_n=\{(x,t)\in A: \tau_A(x,t)=n\}$. Then define $A_\infty=A\setminus (\cup_{n\ge 1} A_n)$. We decompose $\widehat{\Omega}$ into the disjoint union
\[
\widehat{\Omega}=\bigcup_{n=1}^\infty \Omega\times A_n.
\]
Then, on $\widehat{\Omega}_n$, we have
\[
\widehat{T}(\omega,x,t)=(\eta_{x|_n}\omega,\sigma^nx,u+S_nF(x)),
\]
with the exception that $\widehat{T}(\omega,x,t)=(\omega,x,t)$ on $\widehat{\Omega}_\infty$. Since $\mu_A(A_\infty)=0$, we get
\begin{align*}
\int_{\widehat{\Omega}} f(\omega,x,t) \,\widehat{\mathbb{Q}}\circ \widehat{T}^{-1}({\rm d}(\omega,x,t)) = \sum_{n\ge 1} \int_{\widehat{\Omega}_n} f(\eta_{x|_n}\omega,\sigma^nx,u+S_nF(x)) \,\widehat{\mathbb{Q}}({\rm d}(\omega,x,t)).
\end{align*}
For each $n\ge 1$, the mapping $(\omega,x)\to \chi_{B'\times [u,v]}(\eta_{x|_n}\omega,\sigma^nx)$ is $\widetilde{\mathcal{F}}_{k+n}$-measurable. Therefore by \eqref{eq1},
\begin{align*}
I_n:=&\int_{\widehat{\Omega}_n} f(\eta_{x|_n}\omega,\sigma^nx,u+S_nF(x)) \,\widehat{\mathbb{Q}}({\rm d}(\omega,x,t))\\
=&\int_{A_n} \int_{\Omega} f(\eta_{x|_n}\omega,\sigma^nx,u+S_nF(x)) Q_{n+k}(\omega,x) \, \mathbb{P}({\rm d}\omega)\xi({\rm d}(x,t))\\
=&\int_{A_n} \int_{\Omega} Y_{x|_n}(\omega)f(\eta_{x|_n}\omega,\sigma^nx,u+S_nF(x)) Q_{k}(\eta_{x|_n}\omega,\sigma^nx) \, \mathbb{P}({\rm d}\omega)\xi({\rm d}(x,t))\\
=&\int_{A_n} \int_{\Omega} f(\omega,\sigma^nx,u+S_nF(x)) Q_{k}(\omega,\sigma^nx) \, \mathbb{P}({\rm d}\omega)\xi({\rm d}(x,t)),
\end{align*}
where we have used the facts that for $x\in \Sigma$, $Y_{x|_n}$ has mean $1$, the random variable
\[
\omega \to f(\eta_{x|_n}\omega,\sigma^nx,u+S_nF(x)) Q_{k}(\eta_{x|_n}\omega,\sigma^nx)
\]
is independent of $Y_{x|_n}$, and it has the same law as
\[
\omega \to f(\omega,\sigma^nx,u+S_nF(x)) Q_{k}(\omega,\sigma^nx)
\]
under $\mathbb{P}$. By Fubini's theorem,
\begin{align*}
I_n=&\int_{\Omega} \int_{A_n}  f(\omega,\sigma^nx,u+S_nF(x)) Q_{k}(\omega,\sigma^nx) \, \xi({\rm d}(x,t)) \mathbb{P}({\rm d}\omega)\\
=&\int_{\Omega} \int_{A_n}  f(\omega,x,t) Q_{k}(\omega,x) \, \xi\circ \varphi^{-n}({\rm d}(x,t)) \mathbb{P}({\rm d}\omega)\\
=&\int_{\Omega} \int_{A_n}  f(\omega,x,t) Q_{k}(\omega,x) \, \xi\circ \varphi_A^{-1}({\rm d}(x,t)) \mathbb{P}({\rm d}\omega).
\end{align*}
Summing over $n$ we get
\begin{align*}
\int_{\widehat{\Omega}} f(\omega,x,t) \,\widehat{\mathbb{Q}}\circ \widehat{T}^{-1}({\rm d}(\omega,x,t)) =& \int_{\Omega} \int_{A}  f(\omega,x,t) Q_{k}(\omega,x) \, \xi\circ \varphi_A^{-1}({\rm d}(x,t)) \mathbb{P}({\rm d}\omega)\\
=& \int_{\Omega} \int_{A}  f(\omega,x,t) Q_{k}(\omega,x) \, \xi({\rm d}(x,t)) \mathbb{P}({\rm d}\omega)\\
=& \int_{A} \int_{\Omega}  f(\omega,x,t) Q_{k}(\omega,x) \, \mathbb{P}({\rm d}\omega)\xi({\rm d}(x,t)) \\
=&\int_{\widehat{\Omega}} f(\omega,x,t) \,\widehat{\mathbb{Q}}({\rm d}(\omega,x,t)),
\end{align*}
where we have used the fact that $\xi$ is $\varphi_A$-invariant.

To check the ergodicity of $\widehat{\mathbb{Q}}$, let $\widehat{\mathcal{B}}' $ be the semi-algebra of $\widehat{\mathcal{B}} $ consisting of sets of the form
\[
\{(\omega,x,t): x\in[v,u],\ X_{u_1\cdots u_j}(\omega) \in I_{u_1\cdots u_j},\ j=1,\ldots,|u|, \ t\in J\}
\]
for $u,v\in\Lambda^*$, $\{I_{u_1\cdots u_j}\}_{1\le j \le |u|}$ Borel subsets of $[0,\infty)$ and $J\in \mathcal{B}([-1/2,1/2])$. It is clear that $\widehat{\mathcal{B}}'$ generates $\widehat{\mathcal{B}}$, so it is enough to show that for $B_1,B_2\in \widehat{\mathcal{B}}'$ with $\widehat{\mathbb{Q}} (B_1)>0$, $\widehat{\mathbb{Q}} (B_2)>0$, there exists $n_0\in\mathbb N $ such that $\widehat{\mathbb{Q}} (\widehat{T}^{-n_0}B_1\cap B_2)>0$.

The above claim follows from the fact that for $n$ large enough, the events in $\widehat{T}^{-n}B_1$ and $B_2$ on the $\Omega^*$ side are independent. More precisely, for $i=1,2$ write
\[
B_i=\{(\omega,x,t): x\in [v^i,u^i],\ X_{u_1^i\cdots u_j^i}(\omega) \in I_{u_1^i\cdots u_j^i},\ j=1,\ldots,|u^i|, \ t\in J_i\}.
\]
Then for all $n> |u^2|$, the events related to $\{X_u\}_{u\in\Lambda^*}$ in $\widehat{T}^{-n}B_1$ and $B_2$ are independent. This implies that
\[
\widehat{\mathbb{Q}} (\widehat{T}^{-n}B_1\cap B_2)=\Big(\prod_{i=1,2}\prod_{j=1}^{k_i} \mathbb{E}_{\mathbb{P}}(X\mathbf{1}_{\{X\in I_{u_1^i\cdots u_j^i}\}})\Big)\cdot \xi(\varphi_A^{-n}U_1\cap U_2),
\]
where $U_i$ is the projection of $B_i$ to $A$.
From  $\widehat{\mathbb{Q}} (B_1)>0$, $\widehat{\mathbb{Q}} (B_2)>0$ we know that
\[
\prod_{i=1,2}\prod_{j=1}^{|u^i|} \mathbb{E}_{\mathbb{P}}(X\mathbf{1}_{\{X\in I_{u_1^i\cdots u_j^i}\}})>0
\]
and $\xi (U_i)>0$ for $i=1,2$. Since $\xi$ is ergodic, there exists $n>|u^2|$  such that $\xi (\varphi_A^{-n}U_1\cap U_2)>0$, which gives the conclusion. \qed

\subsection{Proof of Lemma \ref{stone}}\label{pflem2}
We shall use the filling scheme, see \cite{Der} for example. For $n\ge 1$ define
\[
G_n=\max_{1\le m \le n} \sum_{k=0}^{m-1} W\circ  \widehat{T}^{k}.
\]
For a function $g$ denote by $g^+=\max(g,0)$ and $g^-=\max(-g,0)$. Then for $n\ge 1$ we have
\[
W=-G_{n+1}^-+G_{n+1}^+-G_n^+\circ\widehat{T}.
\]
Let $G=\lim_{n\to\infty} G_n$. Obviously $G\ge W>-\infty$. Now assume that $\widehat{\mathbb{Q}}(G<\infty)=1$ (By ergodicity this event has $\widehat{\mathbb{Q}}$-mass $0$ or $1$). Then this leads to
\[
W=-G^-+G^+-G^+\circ \widehat{T}.
\]
By construction, since $Q$ is not of $Q_{h_{\nu}(\sigma)}$-type, it holds that ${\widehat{\mathbb{Q}}}(\{W>0\})>0$. By Poincar\'e recurrent theorem, this implies that $\widehat{\mathbb{Q}}$-a.s.,
\[
\kappa=\inf\{n\ge 1: W\circ  \widehat{T}^{n-1}>0\}<\infty.
\]
Since $W\circ \widehat{T}^{\kappa-1}>0$, we have $G\circ \widehat{T}^{\kappa-1}\ge W\circ \widehat{T}^{\kappa-1}>0$, therefore $G^-\circ \widehat{T}^{\kappa-1}=0$, and
\[
W\circ \widehat{T}^{\kappa-1}=G^+\circ \widehat{T}^{\kappa-1}-G^+\circ \widehat{T}^{\kappa}.
\]
Denote by $\widehat{\mathbb{Q}}_{(x,t)}$ the disintegration of $\widehat{\mathbb{Q}}$ w.r.t. $\xi$, which is a probability measure on $\Omega$. For $u\in \mathbb{R}$ and $\xi$-a.e. $(x,t)$ denote by
\[
\phi_u(x,t)=\mathbb{E}_{\widehat{\mathbb{Q}}_{(x,t)}}(e^{iu W\circ \widehat{T}^{\kappa-1}(\cdot,x,t)})
\]
and
\[
\Phi_u(x,t)=\mathbb{E}_{\widehat{\mathbb{Q}}_{(x,t)}}(e^{iu G^+\circ \widehat{T}^{\kappa-1}(\cdot,x,t)}).
\]
Note that under $\widehat{\mathbb{Q}}_{(x,t)}$, $\kappa(\cdot,x,t)$ is a stopping time with respect to the filtration
\[
\{\mathcal{F}_n^{(x,t)}=\sigma(W\circ \widehat{T}^k(\cdot,x,t): 0\le k \le n-1)\}_{n\ge 1}.
\]
Since $G^+\circ \widehat{T}^{\kappa}(\cdot, x,t)$ depends only on the random variables $\{W\circ \widehat{T}^k(\cdot,x,t): k\ge \kappa\}$, it is independent of the stopped $\sigma$-field
\[
\mathcal{F}_\kappa^{(x,t)}=\left\{B\in \sigma\left(\cup_{n\ge 1}\mathcal{F}_n^{(x,t)}\right): B\cap \{\kappa(\cdot,x,t)\le n\}\in \mathcal{F}_n^{(x,t)} \text{ for all } n\ge 1\right \}.
\]
Therefore,
\begin{align*}
\Phi_u(x,t)=&\mathbb{E}_{\widehat{\mathbb{Q}}_{(x,t)}}(e^{iu W\circ \widehat{T}^{\kappa-1}(\cdot,x,t)}e^{iuG^+\circ \widehat{T}^{\kappa}(\cdot,x,t)}) \\
= & \mathbb{E}_{\widehat{\mathbb{Q}}_{(x,t)}}(e^{iu W\circ \widehat{T}^{\kappa-1}(\cdot,x,t)}e^{iuG^+\circ \widehat{T}^{\kappa-1}(\eta_{x|_{\tau_A(x,t)}}\cdot,\varphi_A(x,t))}) \\
= & \mathbb{E}_{\widehat{\mathbb{Q}}_{(x,t)}}(e^{iu W\circ \widehat{T}^{\kappa-1}(\cdot,x,t)} \mathbb{E}_{\widehat{\mathbb{Q}}_{(x,t)}}(e^{iuG^+\circ \widehat{T}^{\kappa-1}(\eta_{x|_{\tau_A(x,t)}}\cdot,\varphi_A(x,t))} | \mathcal{F}_\kappa^{(x,t)})\\
=& \phi_u(x,t)\Phi_u\circ\varphi_A(x,t),
\end{align*}
where the last equality comes from the fact that $\xi$ is $\varphi_A$-invariant and all random variables $X_u$ involved are independent in different generations, and  with expectation $1$. Taking the modulus,  then logarithm we obtain
\[
\log |\Phi_u(x,t)|-\log|\Phi_u\circ \varphi_A(x,t)|=\log |\phi_u(x,t)|.
\]
This implies that $\log  |\phi_u(x,t)|$ is a co-boundary of $(A,\mathcal{B}_A,\varphi_A,\xi)$. Since $ |\Phi_u(x,t)| \vee |\phi_u(x,t)|\le 1$, by Birkhoff ergodic theorem we have
\[
\int_A \log |\phi_u(x,t)| \,\xi(d(x,t))=0.
\]
Since $\log |\phi_u(x,t)|  \le 0$, we deduce that $\log |\phi_u(x,t)| =0$ for $\xi$-a.e. $(x,t)$, which means that $W\circ \widehat{T}^{\kappa-1}$ is a constant. This is a contradiction by considering the case when $W>0$ (and therefore $\kappa=1$) and when $Q$ is not of $Q_{h_{\nu}(\sigma)}$-type. Thus $\widehat{\mathbb{Q}}(G=\infty)=1$. Using similar argument one can also show that $\widehat{\mathbb{Q}}(\inf_{n\ge 1} \sum_{k=0}^{n-1} W\circ  \widehat{T}^{k} =-\infty)=1$. 

\section{Proofs of Theorems~\ref{Application2} and~\ref{Application1}}\label{PFAPPl}

\begin{proof}[Proof of Theorems~\ref{Application2}] For any continuous function $f$ from $\Sigma$ to $\mathbb R$ one has 
$$
\int_\Sigma f(x)\, \nu_p({\rm d}x)=\int_{\pi g \Phi(\Sigma)}\left ( \int_\Sigma f(x) (\nu_p)_{\pi g \Phi,\pi gy}({\rm d}x)\right )\, \nu_p \circ (\pi g \Phi)^{-1}({\rm d}y).
$$
Thus, for all $n\ge 1$, we have 
$$
\int_\Sigma f(x)\, Q_n(\omega,x)\, \nu_p({\rm d}x)=\int_{\pi g \Phi(\Sigma)}\left ( \int_\Sigma f(x) \, Q_n(\omega,x)\, (\nu_p)_{\pi g \Phi,\pi gy}({\rm d}x)\right )\, \nu_p \circ (\pi g \Phi)^{-1}({\rm d}y).
$$
Since we have $h_X<\dim_H((\nu_p)_{\pi g \Phi,\pi g\Phi(x)})=h_{p,\pi}(\mathcal{P},\Phi)$, the same arguments as those used in the proof of Corollary~\ref{cor-2.11} (see Section~\ref{pfcor2}) imply that  
$$
\int_\Sigma f(x)\, Q\cdot\nu_p({\rm d}x)=\int_{\pi g \Phi(\Sigma)}\left ( \int_\Sigma f(x) \, Q\cdot (\nu_p)_{\pi g \Phi,\pi gy}({\rm d}x)\right )\, \nu_p \circ (\pi g \Phi)^{-1}({\rm d}y)
$$
This implies the claim about absolute continuity.

For the exact dimensionality, again since $h_X<\dim_H((\nu_p)_{\pi g \Phi,\pi g\Phi(x)})$, by Corollary \ref{co-ed} we have that for $\nu_p$-a.e. $x\in\Sigma$, (or equivalently, for $\nu_p\circ (\pi g \Phi)^{-1}$-a.e. $\pi g\Phi(x)$), $Q\cdot (\nu_p)_{\pi g \Phi,\pi g\Phi(x)}$ is exact-dimensional with dimension $h_{p,\pi}(\mathcal{P},\Phi)-h_X$. If $\mathcal{I}$ satisfies OSC, one can easily deduce that $Q\cdot (\nu_p)_{\pi g \Phi,\pi g\Phi(x)}\circ \Phi^{-1}$ is also exact-dimensional with dimension
\[
\frac{h_{p,\pi}(\mathcal{P},\Phi)-h_X}{\chi_{p,r}}
\]
(to do so, use the fact that $\mathcal I$ satisfies the strong open set condition~\cite{Sch94}, and then the fact that the contraction ratios are not random to apply  \cite[Lemma 8.8]{BiHaJo2011}; this makes it possible to sandwich, at $Q\cdot\nu_p$-almost every point   $x$, the ball $B(\Phi(x),r)$ by images of cylinders centered at $x$, and with diameters comparable to $r$). Note that the Lyapunov exponent remains the same since
\[
\mathbb{E}(\sum_{i\in\Lambda} -p_iX_i\log r_i)=\sum_{i\in\Lambda} -p_i\log r_i.
\]
Now, together with the fact that $(Q\cdot \nu_p)\circ (\pi g\Phi)^{-1}$ is absolutely continuous with respect to $\nu_p\circ (\pi g \Phi)^{-1}$, we can get that $\nu_p\circ (\pi g \Phi)^{-1}$ is exact-dimensional with dimension
\[
\frac{h_p-h_{p,\pi}(\mathcal{P},\Phi)}{\chi_{p,r}}.
\]
Indeed, we deduce from a theorem by Marstrand (see \cite[Theorem 5.8]{Falconer1985}) that the lower local dimension of $(Q\cdot \nu_p)\circ \Phi^{-1}$ is at least $\dim ((Q\cdot \nu_p)\circ (\pi g\Phi)^{-1})+\dim (Q\cdot (\nu_p)_{\pi g \Phi,\pi g\Phi(x)}\circ \Phi^{-1})$, i.e. at least
\[
D=\frac{h_p-h_{p,\pi}(\mathcal{P},\Phi)}{\chi_{p,r}}+\frac{h_{p,\pi}(\mathcal{P},\Phi))-h_X}{\chi_{p,r}}=\frac{h_{p}-h_X}{\chi_{p,r}},
\]
at $(Q\cdot \nu_p)\circ \Phi^{-1}$-almost every $y\in\mathbb{R}^d$. Moreover, the proof of \cite[Theorem 3.1(i)]{FJ14} establishes, without additional moment assumption on $X$ that the upper local dimension of $(Q\cdot \nu_p)\circ \Phi^{-1}$ is at most $D$ at $(Q\cdot \nu_p)\circ \Phi^{-1}$-almost every $y\in\mathbb{R}^d$. To see this, note that in this proof we need a weaker property than \cite[Lemma 3.4(iii)]{FJ14}, namely  $\limsup_{n\to\infty} \frac{1}{n}\log \|\mu^{[x|_n]}\| \le 0$ for $Q\cdot\nu_p$-almost every~$x$, where $\mu^{[x|_n]}$ is the copy of $Q\cdot\nu_p$ associated with the subtree rooted at~$x|_n$. But this property follows easily from the integrability of $Z=\|Q\cdot\nu_p\|$. The control of the upper local dimension of $(Q\cdot \nu_p)\circ \Phi^{-1}$ can also be obtained by using an approach similar to that used in the proof of Theorem \ref{thm1} (c), by using the natural covering of $K$ generated by the iterations of the IFS $\mathcal{I}$. Eventually, $(Q\cdot \nu_p)\circ \Phi^{-1}$ is exact dimensional with dimension~$D$. 
\end{proof}

\begin{proof}[Proof of Theorems~\ref{Application1}]
Since $h_X<\dim \nu^y$ for ${\pi_2}_*\nu $-almost every $y$, it follows from Theorem~\ref{thm1} that for ${\pi_2}_*\nu $-almost every $y$, the action of $Q$ on $\nu^y$ is full, and consequently, using the disintegration of~$\nu$ as ${\pi_2}_*\nu({\rm d}y) \nu^y({\rm d}x)$, and the same lines as in the proof of Corollary~\ref{cor-2.11}, we see that $Q$ acts fully on $\nu$, and ${\pi_2}_*(Q\cdot\nu)$ is absolutely continuous with respect to its expectation ${\pi_2}_*\nu$, with density the total mass of $ Q\cdot \nu^y$ for ${\pi_2}_*\nu $-almost every $y$. Moreover, conditional on $Q\cdot \nu^y\neq 0$, $Q\cdot \nu^y$ is exact dimensional with dimension  $h_\nu(\sigma)-h_{{\pi_2}_*\nu}(\sigma_2)-h_X$ with respect to $d$.  
 
\medskip

The result about the exact dimensionality of $Q\cdot \nu$ is now a direct consequence of the  following lemma. 
\begin{lemma}\label{lll}

Let $m$ be a Borel probability measure on $(\Lambda^{\mathbb Z_+},d)$. Suppose that $m$ is exact dimensional with dimension $\delta$ with respect to the metric $d$. Denote by $\delta_2$ the lower Hausdorff dimension of ${\pi_2}_*m$ with respect to the metric induced by $d$, and let $\underline \delta_1$ and $\overline \delta_1$ be the essential infimum and the essential supremum of the lower Hausdorff dimensions of the conditional measures $m^y$ with respect to $d$ again, where $m^y$ is obtained from the disintegration of $m$ with respect to ${\pi_2}_*m$. Then, with respect to the metric $d_{(b_1,b_2)}$, for $m$-almost every point $z$, we have 
\begin{equation}\label{dimbounds}
\frac{\overline \delta_1}{\log(b_1)}+\frac{\delta_2}{\log(b_2)}\le \underline{\dim}_{{\rm loc}}(m,z)\le \overline{\dim}_{{\rm loc}}(m,z)\le \frac{\delta}{\log(b_2)}-\Big (\frac{1}{\log(b_2)}-\frac{1}{\log(b_1)}\Big )\underline \delta_1.
\end{equation}
So, if $\underline \delta_1=\overline \delta_1$ and  $\delta=\underline \delta_1+\delta_2$, then measure $m$ is exact dimensional. 
\end{lemma}
The first inequality in \eqref{dimbounds} follows from a result of Marstrand (see \cite[Theorem 5.8]{Falconer1985}), while the second one  can be deduced from the proof of Theorem 2.11 in \cite{FengHu2009}. 
\end{proof}

\noindent {\bf Acknowledgement. }We thank A.-H. Fan for his comments on a first version of this work, especially for pointing to us the characterisation of the action of $Q_\alpha$ on Borel probability measures. We would also like to thank two anonymous referees for their useful comments to help improving the paper.


\begin{thebibliography}{0}

 \bibitem{BF2018}
 Barral, J., and D.-J. Feng.
 ``Projections of random Mandelbrot measures."
 \textit{Adv. Math.} 325 (2018): 640--718. 

\bibitem{BF2020}
Barral, J., and D.-J. Feng.
``Dimensions of random statistically self-affine Sierpinski sponges in $\mathbb R^k$."
\textit{J. Math. Pures Appl.} 149(2021): 254--303.

\bibitem{BM}
Barral, J., and B.\,B.~Mandelbrot.
``Multifractal products of cylindrical pulses."
\textit{Probab. Theory Relat. Fields} 124 (2002): 409--430.

\bibitem{BM2}
Barral, J., and  B. B. Mandelbrot.
``Random Multiplicative Multifractal Measures, I, II, III."
in M.~Lapidus, M.~van Frankenhuijsen, eds., Fractal geometry and applications: a jubilee of Beno\^\i t Mandelbrot, \textit{Proc. Symp. Pure Math.} Vol.72, No.2 (2004):3--90.

\bibitem{BJ}
Barral, J., and X. Jin.
"On exact scaling log-infinitely divisible cascades."
\textit{Probab. Theory Relat. Fields} 160 (2014): 521--565. 

\bibitem{BaMu}
Bacry, E., and J.-F.~Muzy.
``Log-infinitely divisible multifractal processes."
\textit{Commun.\ Math.\ Phys.} 236 (2003): 449--475.
  
\bibitem{BKM}
Barany, B., A. K\"aenm\"aki and I.~Morris.
``Domination, almost additivity, and thermodynamic formalism for planar matrix cocycles."
\textit{Isra\"el J. Math.} 239(2020):173--214. 
  
\bibitem{BenNasr}
Ben Nasr, F.
``Mesures al\'eatoires de Mandelbrot associ\'ees \`a des substitutions."
\textit{C. R. Acad. Sci. Paris S\'er. I Math} 304, no.10 (1987): 255--258.

\bibitem{BeWeWo}
Berestycki, N., C. Webb and M.-D. Wonk.
``Random Hermitian matrices and Gaussian multiplicative chaos."
\textit{Probab. Th. Relat. Fields} 172 (2018): 103--189.  

\bibitem{Biggins77}
Biggins, J. D.
``Margingale convergence in the branching random walks."
\textit{J. Appl. Prob.} 14 (1977): 25--37.

\bibitem{BiHaJo2011}
Biggins, J. D., B.~M.~Hambly and O.~D.~Jones.
``Multifractal spectra for random self-similar measures via branching processes."
\textit{Adv. Appl. Prob.} 43 (2011): 1--39. 

\bibitem{BK04}
Biggins, J. D., and A. E.~Kyprianou.
``Measure change in multitype branching."
\textit{Adv. Appl. Prob.} 36 (2004): 544--581.

\bibitem{Billard1965}
Billard, P.
``S\'eries de Fourier al\'eatoirement born\'ees, continues, uniform\'ement convergentes."
\textit{Ann. Scient. Ec. Norm. Sup.} 82 (1965): 131--179.

\bibitem{Bog1999}
Bogensch\"utz, T.
``Entropy, pressure, and a variational principle for random dynamical systems."
\textit{Random Comput., Dynam.} 1 (1992/1993): 99-116.

\bibitem{Breiman1968}
Breiman, L.
\textsl{Probability.} Addison-Wesley, Reading, Massachusetts, 1968.

\bibitem{Dek}
Dekking, F. M.
``On transience and recurrence of Generalized random walks."
\textit{Z. Wahrsch. Verw. Gebiete} 61 (1982): 459--465.

\bibitem{Der}
Derriennic, Y.
``Ergodic theorem, reversibility and the filling scheme."
\textit{Colloq. Math.} 118, no. 2 (2010): 599--608.

\bibitem{DuSh}
Duplantier, B., and S. Sheffield.
``Liouville quantum gravity and KPZ"
\textit{Inv. Math.} 185 (2011): 333-393.

\bibitem{Dur}
Durrett, R.
\textit{Probability: Theory and Examples, 5th ed.} Cambridge Series in Statistical and Probabilistic Mathematics, Cambridge University Press, Cambridge, 2019.

\bibitem{DuLi83}
Durrett, R., and T.~Liggett.
``Fixed points of the smoothing transformation."
\textit{Z. Wahrsch. Verw. Gebiete} 64 (1983): 275--301.

\bibitem{Falconer1985}
Falconer, K. J.
\textit{The geometry of fractal sets.} Cambridge University Press. Cambridge, 1985. 

\bibitem{FJ14}
Falconer, K. J., and X. Jin.
``Exact dimensionality and projections of random self-similar measures and sets."
\textit{J. Lond. Math. Soc. (2)} 90 (2014): 388--412.

\bibitem{FJ} K. Falconer and X. Jin, ``Exact dimensionality and projection properties of Gaussian multiplicative chaos measures", {\em Trans. Amer. Math. Soc.},  \textbf{372}(4) (2019): 2912--2957. 

\bibitem{Fan1989c}
Fan, A.-H.
\textit{Recouvrement al\'eatoire et d\'ecompositions de mesures}, Publication d'Orsay, Paris,1989.

\bibitem{Fan1990}
Fan, A.-H.
``Sur quelques processus de naissance et de mort."
\textit{C.~R.~Acad.\ Sci. Paris S\'er.I.} 310 (1990): 441--444.

\bibitem{Fan1994}
Fan, A.-H.
``Sur les dimensions de mesures."
\textit{Studia Math.} 111 (1994): 1--17.

\bibitem{Fan1997a}
Fan, A.-H.
``Sur le chaos de L\'evy d'indice $0<\alpha <1$."
\textit{Ann. Sci. Math. Qu\'ebec} 21 (1997): 53--66.

\bibitem{Fan}
Fan, A.-H.
``On Markov-Mandelbrot martingales."
\textit{J. Math. Pures Appl.} 81 (2002): 967--982.

\bibitem{Fan2004}
Fan, A.-H.
``Some topics in the theory of multiplicative chaos."
\textit{Progress in Probability} 57 (2004): 199--134. 

\bibitem{FanJiang2001}Fan, A.-H., and Jiang, J.
``On Ruelle-Perron-Frobenius Operators. II. Convergence speeds.''
 \textit{Comm. Math. Phys.} 223 (2001): 143--159. 

\bibitem{FengHu2009}
Feng, D.-J., and H. Hu.
``Dimension theory of iterated function systems."
\textit{Comm. Pure. Appl. Math} 62 (2009): 1435--1500. 

\bibitem{Fur}
Furstenberg, H.
``Ergodic fractal measures and dimension conservation."
\textit{Ergod. Th. \& Dynam. Sys.} 28 (2008): 405-422.

\bibitem{Heurteaux1998}
Heurteaux, Y.
``Estimation de la dimension inf\'erieure et de la dimension sup\'erieure des mesures."
\textit{Ann. Inst. H. Poincar\'e} 34 (1998): 309--338. 

\bibitem{Hofbauer1977}
Hofbauer, F.
``Examples for the nonuniqueness of the equilibrium state."
\textit{Trans. Amer. Math. Soc.} 228 (1977): 223--241. 

\bibitem{Kahane1985a}
Kahane, J.-P.
\textit{Some random series of functions}, Cambridge University Press, Cambridge, 1985.

\bibitem{Kahane1985}
Kahane, J.-P.
``Sur le chaos multiplicatif."
\textit{Ann. Sci. Math. Qu\'ebec} 9 (1985): 105--150.

\bibitem{Kahane1987}
Kahane, J.-P.
``Chaos multiplicatif et dimension de Hausdorff."
\textit{Ann. Inst. H. Poincar\'e Probab. Stat.} 23 (1987): 289--296.

\bibitem{Kahane1987b}
Kahane, J.-P.
``Positive martingales and random measures."
\textit{Chin. Ann. Math.} 8B, no. 1 (1987): 1--12.

\bibitem{KP1976}
Kahane, J.-P., and J.~Peyri\`ere.
``Sur certaines martingales de Mandelbrot."
\textit{Adv. Math.} 22 (1976): 131--145.

\bibitem{Kolmogorov1961}
Kolmogorov, A. N.
``Pr\'ecisions sur la structure locale de la turbulence dans un fluide visqueux aux nombres de Reynolds \'elev\'es, M\'ecanique de la trubulence."
\textit{Colloq.\ intern.} CNRS Marseille 1961, Editions CNRS 1962, pp. 447--451.

\bibitem{KRV}
Kupiainen, A., R. Rhodes and V. Vargas.
``Integrability of Liouville theory: proof of the DOZZ Formula."
\textit{Ann. Math.} 191 (2020): 81-166.

\bibitem{Ledrappier1977}
Ledrappier, F.
``Un exemple de transition de phase."
\textit{Monat. Math.} 83 (1977): 147--153.

\bibitem{Liv1996}
Liverani, C.
``Central limit theorem for deterministic systems.''
International Conference on Dynamical Systems (Montevideo, 1995), vol. 362 of Pitman Res. Notes Math. Ser., pp 56--75. Longman, Harlow.

\bibitem{Lyons1990}
Lyons, R.
``Random walks and percolation on trees."
\textit{Ann. Probab.} 18 (1990): 931--958.

\bibitem{Lyons97}
Lyons, R.
``A simple path to Biggins' martingale convergence for branching random walk."
In \textit{Classical and Modern Branching Processes} (IMA Vol. Math. Appl. 84), eds. K. B. Athreya and P. Jagers, Springer, New York, 1997, pp. 217--222.

\bibitem{MaWen}
Ma, J., and Z. Wen.
``Hausdorff and Packing Measure of Sets of Generic Points: A Zero-Infinity Law."
\textit{J. London Math.Soc.} 69 (2004): 383--406.

\bibitem{Mandelbrot1971}
Mandelbrot, B. B.
``Possible refinement of the log-normal hypothesis concerning the distribution of energy dissipation in intermittent turbulence in statistical models and turbulence."
Symposium at U.~C.~San Diego 1971, Lecture Notes in Physics, Springer-Verlag 1972, pp. 333--351.

\bibitem{Mandelbrot1974}
Mandelbrot, B. B.
``Intermittent turbulence in self-similar cascades: divergence of hight moments and dimension of the carrier."
\textit{J. Fluid. Mech.} 62 (1974): 331--358.

\bibitem{Ngai}
Ngai, S.-M.
``A dimension result arising from the $L^q$-spectrum of a measure."
\textit{Proc. Amer. Math. Soc.} 125 (1997): 2943--2951.


\bibitem{RSV}
Rhodes, R., J. Sohier and V. Vargas.
``L\'evy multiplicative chaos and star scale invariant random measures."
\textit{Ann. Probab.} 42 (2014): 689--724.

\bibitem{Sch94}
Schief, A.
``Separation properties for self-similar sets."
\textit{Proc. Amer. Math. Soc.} 122 (1994): 111--115. 



\bibitem{WW95}
Waymire, E. C., and S. C.~Williams, 
``Multiplicative cascades: dimension spectra and dependence."
in Proceedings of the Conference in Honor of Jean-Pierre Kahane (Orsay, 1993). \textit{J.~Fourier Anal.\ Appl.}  Special Issue 1995, pp. 589--609.

\bibitem{WW96}
Waymire, E. C., and S.\,C. Williams.
``A cascade decomposition theory with applications to Markov and exchangeable cascades."
\textit{Trans. Amer. Math. Soc.} 348 (1996): 585--632. 

\bibitem{Walters}
Walters, P.
``Invariant measures and equilibrium states for some mappings which expand distances."
\textit{Trans. Amer. Math. Soc.} 236 (1978): 121--153.

\bibitem{Yuri00}
Yuri, M.
``Weak Gibbs measures for certain non-hyperbolic systems."
\textit{Ergod. Th. $\&$ Dynam. Sys.} 20 (2000): 1495--1518.

\end{thebibliography}
\end{document}